\documentclass{elsarticle}
\usepackage{latexsym,amsfonts,amssymb,amsmath,euscript,amsthm}
\usepackage{graphicx} 
\usepackage{hyperref}         
\usepackage{color,fancybox}
\usepackage{float}

\newcommand{\R} {\mathbb{R}}
\newcommand{\N} {\mathbb{N}}

\newcommand{\eps}{\epsilon}

\def\accentsfrancais{applemac}
  \RequirePackage[\accentsfrancais]{inputenc}

\begin{document}

\begin{frontmatter}

\title{Straight rod with different order of thickness}

\author{G. Griso$^{a}$, M. Villanueva-Pesqueira$^{b}\footnotemark[1]$
{\footnotesize
\begin{center}
$^{a}$ Laboratoire J.-L. Lions--CNRS, Bote courrier 187, Universit  Pierre et
Marie Curie,\\ 4~place Jussieu, 75005 Paris, France, \; Email: griso@ljll.math.upmc.fr\\
$^{b}$ Departamento de Matem\'atica Aplicada, Facultad de Matem\'aticas, Universidad Complutense de Madrid, \\
28040 Madrid, Spain, Email: manuelvillanueva@mat.ucm.es
\end{center} 
}
}

\begin{abstract}
%
 \noindent In this paper, we consider rods whose thickness vary linearly between $\eps$ and $\eps^2$.  Our aim is to study the asymptotic behavior of these rods in the framework of the linear elasticity. We use a decomposition method of the displacement fields of the form $u=U_e + \bar{u}$, where $U_e$ stands for  the translation-rotations of the cross-sections and $\bar{u}$ is  related to their deformations.  We establish a priori estimates. Passing to the limit  in a fixed domain gives the problems satisfied by the bending, the stretching and the torsion limit fields which are ordinary differential equations depending on  weights.
\end{abstract}

\begin{keyword}
Linear elasticity\sep Rods
\MSC 74B05 \sep 74K10

\end{keyword}

\end{frontmatter}


\numberwithin{equation}{section}
\newtheorem{theorem}{Theorem}[section]
\newtheorem{lemma}[theorem]{Lemma}
\newtheorem{corollary}[theorem]{Corollary}
\newtheorem{proposition}[theorem]{Proposition}
\newtheorem{definition}[theorem]{Definition}
\newtheorem{remark}[theorem]{Remark}
\allowdisplaybreaks
\def \Ge{{\bf e}}
\def\ds{\displaystyle}

\footnotetext[1]{The second author was partially supported by grant MTM2012-31298, MINECO, Spain and Grupo de Investigaci\'on CADEDIF, UCM and by a FPU fellowship (AP2010-0786) from the Goverment of Spain.}

\section{Introduction}
In this paper we are interested in analyzing the asymptotic behavior of a thin rod with different order of thickness in the framework of the linear elasticity. We consider a straight rod of fixed length where the cross-sections are  bounded Lipschitz domains with small diameter of order varying between $\eps$ and $\eps^2$. To be more precise, the order of the thickness of the rod is given by $\eps\rho_\eps(\cdot)$ where $\rho_\eps(\cdot)$ is a linear function depending on the cross-section of the rod such that it is $1$ at the bottom and $\eps$ at the top of the rod. We investigate how the variable thickness of the rod affects to the a priori estimates and the limit problems.

Since the diameter of the rod tends to zero, this work belongs to the field of elliptic problems posed on thin domains. Many fields of science involve the study of thin domains, for example in solid mechanics (thin rods, plates, shells), fluid dynamics (lubrication, meteorology problems, ocean dynamics), physiology (blood circulation), etc. There are many papers
dedicated to the study of the thin structures from the point of view of the elasticity, see e.g. \cite{TV96, SHSP99} for models of rods and \cite{CI97, CI00} for plates and shells.

Our work is based on the decomposition of a displacement of the rod according to \cite{Gr08}. Every displacement of the rod is the sum of an elementary displacement, it characterizes the translation and the rotation of the cross-sections, and a warping which is the residual displacement related to the deformation of the cross-section. This decomposition of the rod was introduced in  \cite{Gr05} and \cite{Gr05A} and it allows to obtain the Korn inequality  as well as the asymptotic behavior of the strain tensor of a sequence of displacements in a simple and effective way.

The notion of the elementary displacement together with the unfolding method (see \cite{CDG02, CDG08}) has led to a new method in elasticity which has been successfully applied to many problems, see e.g \cite{BGG07, BGG07A, BG08} and \cite{Gr03,Gr04, Gr05, Gr05A, Gr05B, Gr06, Gr08}. References and other applications of the unfolding operator 
technique can be found in \cite{D06,DP09,CDGZ12, AV14}.

Our paper is organized as follows. In Section 2 we describe the geometry of the rod, introduce the decomposition of a displacement field of the rod and we give some estimates of the decomposition fields in terms of the strain energy (Theorem \ref{a priori estimates}).  The proof of the Theorem \ref{a priori estimates} is based on the approximation of the displacement of the rod by a rigid body displacement. Of course, the estimates may depend on the function $\rho_\eps(\cdot).$

Section 3 is dedicated to get a priori estimates for the different fields assuming that the rod is clamped at the bottom. These estimates have an essential importance in our study to pass to the limit. Moreover, we introduce the rescaling operator $\Pi_\eps$ which allows to work in a fixed domain. One particular feature of this transformation is that the ratio of the dilation of the fixed rod depends on the third variable, it is given by the function $\eps \rho_\eps(\cdot)$. Then a special care is dedicated to the estimate of the derivatives with respect to the third variable.

In Section 4 we give the limit of the displacements and we show a few relations between some of them. Since some of the a priori estimates established depend on the variable thickness $\rho_\eps(\cdot)$ we introduce some weighted Sobolev spaces which allow to obtain the limit fields in a natural way. In Section 5 we pose the problem of elasticity and we 
specify the assumptions on the applied forces. We show that the choice of the applied forces is reasonable to get the suitable estimate of the total elastic energy, so that the convergence results of the previous sections can be used. 
In Section 6 we derive the equations satisfied by the limit fields and we prove the strong convergence of the energy. Moreover, we deduce some strong convergences of the fields of the displacement's decomposition. Finally, in Section 7 we summarize the main results.

\section{Decomposition of the displacement of a straight rod with different order of thickness} \label{DEC}

Let $\omega$ be a bounded domain in $\R^2$ with Lipschitzian boundary, diameter equal to $R$ and star-shaped with respect to a disc of radius $R_1$. We choose the origin $O$ of coordinates at the center of gravity of $\omega$ and we choose as coordinates
axes $(O; \Ge_1)$ and $(O; \Ge_2)$ the principal axes of inertia of $\omega$. Notice that, with this reference frame we have
\begin{equation}\label{integral}
\int_{\omega} x_1\, dx_1dx_2 = \int_{\omega} x_2\, dx_1dx_2 = \int_{\omega} x_1x_2\, dx_1dx_2=0.
\end{equation}

The cross-section $\omega_{\eps, x_3}$ of the rod is obtained by transforming $\omega$ with a dilatation of center $O$ and ratio $\eps \rho_{\eps}(x_3),$ where
$$\ds \rho_{\eps}(x_3)= 1-\frac{x_3}{L}\big(1-\frac{\eps}{L}\big),\qquad x_3\in [0,L].$$ We assume $0<\eps<L/2$ and $0<R_1<1/2$ without loss of generality.

\begin{definition} 
The straight rod is defined as follows:
$$\Omega_\eps = \{ x=(x_1, x_2, x_3) \in \R^3 \; | \; x_3 \in (0, L), (x_1, x_2) \in \omega_{\eps, x_3} \},$$
where 
$\ds \omega_{\eps, x_3} = \Big\{ (x_1, x_2) \in \R^2 \;\; | \;\; \Big(\frac{x_1}{\eps \rho_{\eps}(x_3)}, \frac{x_2}{\eps \rho_{\eps}(x_3)}\Big) \in \omega\Big\}=\eps\rho_{\eps}(x_3)\omega$.
\end{definition}

\begin{figure}[H]
  \centering
    \includegraphics [width=0.7\textwidth]{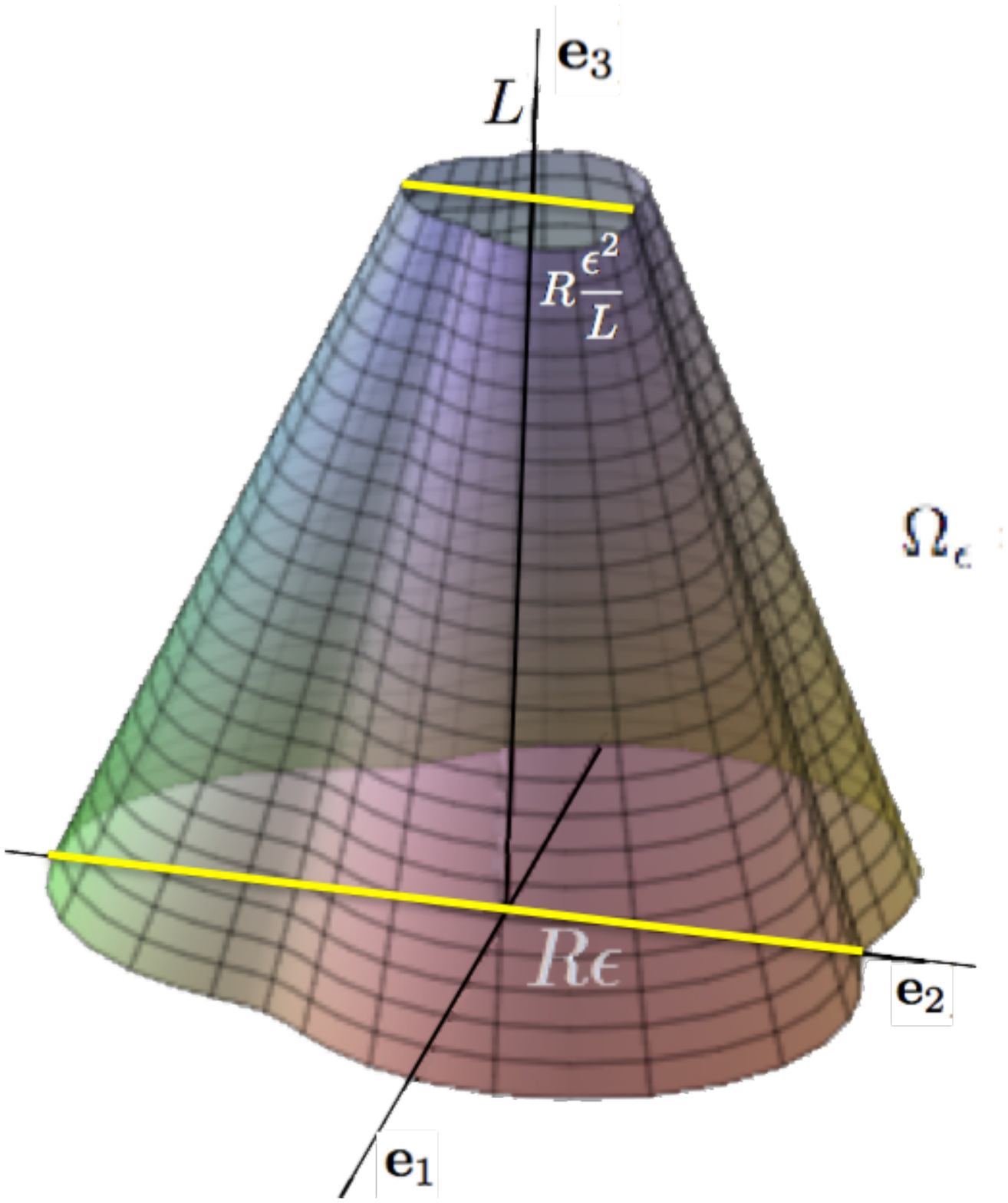}
    \caption{Straight rod $\Omega_\eps$. }
    \label{rod}
\end{figure}
Notice that the center line of the straight rod is the coordinate axis $(O; {\Ge}_3)$. Moreover, the thickness of the thin rod depends on $x_3$, it is given by the function $\ds \eps \rho_\eps(x_3)=
\eps -\frac{x_3}{L}\eps\big(1- \frac{\eps}{L}\big)$. Observe that the diameter of the lower boundary is order $\eps$ while the diameter of the upper boundary is order $\eps^2/L$. (See Figure \ref{rod}.)
\smallskip

Now, we define an elementary displacement associated to a displacement of the rod.

\begin{definition} The elementary displacement $U_e$, associated to $u \in L^1(\Omega_\eps; \R^3),$ is given by:
$$ U_e(x) = \mathcal{U}(x_3)+ \mathcal{R}(x_3) \land (x_1{\Ge}_1+x_2{\Ge}_2), \quad x \in \Omega_\eps,$$
where for a.e. $x_3\in (0,L)$ 
\begin{equation} \label{Decomp}
\left\{
\begin{aligned}
 \mathcal{U} (x_3) &= \frac{1}{|\omega| \rho_\eps(x_3)^2\eps^2}\int_{\omega_{\eps, x_3}} u(x_1, x_2, x_3) \, dx_1 dx_2,  \\
 \mathcal{R}_3(x_3) &= \frac{1}{(I_1 + I_2)\rho_\eps(x_3)^4\eps^4} \int_{\omega_{\eps, x_3}} \big[(x_1{\Ge}_1+x_2{\Ge}_2)  \land u(x_1, x_2, x_3)\big] \cdot  {\Ge}_3\, dx_1 dx_2, \\
 \mathcal{R}_\alpha(x_3) &= \frac{1}{(I_{3 - \alpha})\rho_\eps(x_3)^4\eps^4} \int_{\omega_{\eps, x_3}} \big[(x_1{\Ge}_1+x_2{\Ge}_2)\land u(x_1, x_2, x_3)\big] \cdot  {\Ge}_\alpha\, dx_1 dx_2,  \\
  I_\alpha & = \int_{\omega} x_{\alpha}^2 \, dx_1dx_2, \, \textrm{ for } \alpha \in \{1, 2\}. 
 \end{aligned}
\right. 
\end{equation}
\end{definition}
The first component $\mathcal{U}$ of $U_e$  is the displacement of the center line. The second component $\mathcal{R}$ represents the rotation of the cross-section. Under the action
of an elementary displacement the cross-section $\omega_{\eps, x_3}$ is translated by $\mathcal{U}(x_3)$ and it is rotated around the vector $\mathcal{R}(x_3)$ with an angle $\|\mathcal{R}(x_3)\|_2$, where $\|\cdot\|_2$ is the Euclidean norm in $\R^3$. Observe that, the torsion of the rod is given by the displacement $\mathcal{R}_3(x_3){\Ge}_3 \land (x_1{\Ge}_1+x_2{\Ge}_2)$.

Any displacement $u$ of the rod can be decomposed as 
\begin{equation}\label{Decomp1}
u= U_e + \bar{u}.
\end{equation}
The displacement $\bar{u}$ is the warping.

Next theorem gives estimates of the components of the elementary displacement $U_e$ and of the warping $\bar{u}$ in terms of $\eps$, $\rho_\eps$ and of the strain energy of the displacement $u$. Notice that if $u$ belongs to $H^1(\Omega_\eps)$ the functions $\mathcal{U}$ and $\mathcal{R}$ belong to $H^1((0,L); \R^3)$.
\begin{theorem}\label{a priori estimates}
Let $u \in H^1(\Omega_\eps; \R^3)$ and $u =U_e + \overline{u}$ the decomposition given by (\ref{Decomp})-(\ref{Decomp1}). Then the following estimates hold:

\begin{equation} \label{estimates}
\left\{
\begin{gathered}
\Big\| \frac{\overline{u}}{\rho_\eps} \Big\|_{L^2(\Omega_{\eps}; \R^3)} \leq C \eps \| (\nabla u)_{\mathcal{S}} \|_{[L^2(\Omega_{\eps})]^9},\\
 \Big\| \rho_\eps\, \Big(\frac{d \mathcal{U}}{d x_3} - \mathcal{R} \land \Ge_{3}\Big)  \Big\|_{L^2((0, L); \R^3)} \leq \frac{C}{\eps} \| (\nabla u)_{\mathcal{S}} \|_{[L^2(\Omega_{\eps})]^9},\\
 \Big \| \rho^2_\eps\frac{d\mathcal{R}}{d x_3}  \Big\|_{L^2((0, L); \R^3)}  \leq \frac{C}{\eps^2} \| (\nabla u)_{\mathcal{S}} \|_{[L^2(\Omega_{\eps})]^9},\\
  \| \nabla\bar{u} \|_{[L^2(\Omega_{\eps}; \R^3)]^9} \leq C \| (\nabla u)_{\mathcal{S}} \|_{[L^2(\Omega_{\eps})]^9}.
\end{gathered}
\right. 
\end{equation}

\noindent The constants are independent of $\eps$ and $L$.
\end{theorem}

\begin{proof}To prove the above estimates we are going to introduce a partition of the rod $\Omega_\eps$ in several small portions where every of these small rods are star-shaped with respect to suitable balls 
which verify that the ratio between their radius and the diameters of the portions remains uniformly bounded. Then we use the approximation of the displacement $u$
by a rigid body displacement in each portion, (see Theorem 2.3 in \cite{Gr08}).  
\medskip

{\it Step 1. } Construction of the partition.
\medskip

 We start by considering the first portion of the rod
$$\Omega^0_{\eps}= \{ x=(x_1, x_2, x_3) \in \R^3 \; | \; x_3 \in (0, \eps), (x_1, x_2) \in \omega_{\eps, x_3} \}.$$

First, notice that $\Omega^0_{\eps}$ has a diameter less than $(R+1)\eps$ and all the cross-sections of $\Omega^0_{\eps}$ are  star-shaped with respect to a disc of radius $R_1\eps \rho_\eps(\eps)$.
Therefore, by a simple geometrical argument, it is easy to check that this portion is star-shaped with respect to a ball of radius $R_1\eps\rho_\eps(\eps)$.
\smallskip

We consider now a partition of the interval [0,L] defined as
$$s^0_{\eps}= 0 < s^1_{\eps}= \eps < s^2_{\eps}=  s^1_{\eps} +\eps\rho_\eps(s^1_{\eps})< \dots <  s^{N_\eps}_{\eps}=  s^{N_\eps-1}_{\eps} + \eps\rho_\eps(s^{N_\eps-1}_{\eps})\leq L \leq s^{N_\eps+1}_{\eps}=  s^{N_\eps}_{\eps} + \eps\rho_\eps(s^{N_\eps}_{\eps}) .$$
Hence, the points of the partition $\{s^k_\eps\}$ are the elements of an arithmetico-geometric sequence
$$s^k_\eps= \eps \frac{1-\rho_\eps(\eps)^k}{1-\rho_\eps(\eps)} \quad \Longrightarrow \quad \lim_{k\to \infty} s^k_\eps = \frac{\eps}{1-\rho_\eps(\eps)} = \frac{L}{\ds 1-\frac{\eps}{L}}>L.$$
It makes sense to define $N_\eps$ as the largest integer such that $s^{N_\eps}_{\eps} \leq L$.


\smallskip

The $(k+1)$-portion of the rod is defined as 
$$\Omega^k_{\eps} = \{ x \in \R^3 \; | \; x_3 \in \big(s^k_\eps, s^k_\eps+\eps\rho_\eps(s^k_\eps)\big), (x_1, x_2) \in \omega_{\eps, x_3} \}, \quad 0\leq k \leq N_{\eps} -2,$$
and
$$\Omega^{N_{\eps}-1}_\eps = \{ x \in \R^3 \; | \; x_3 \in \big(s^{N_\eps-1}, L\big), (x_1, x_2) \in \omega_{\eps, x_3} \}.$$



Therefore, we obtain
$$\Omega_\eps = \textrm{Int}\Big\{\bigcup_{k=0}^{N_\eps-1} \overline{\Omega^k_{\eps}}\Big\}.$$

\begin{figure}[H]
  \centering
    \includegraphics [width=0.7\textwidth]{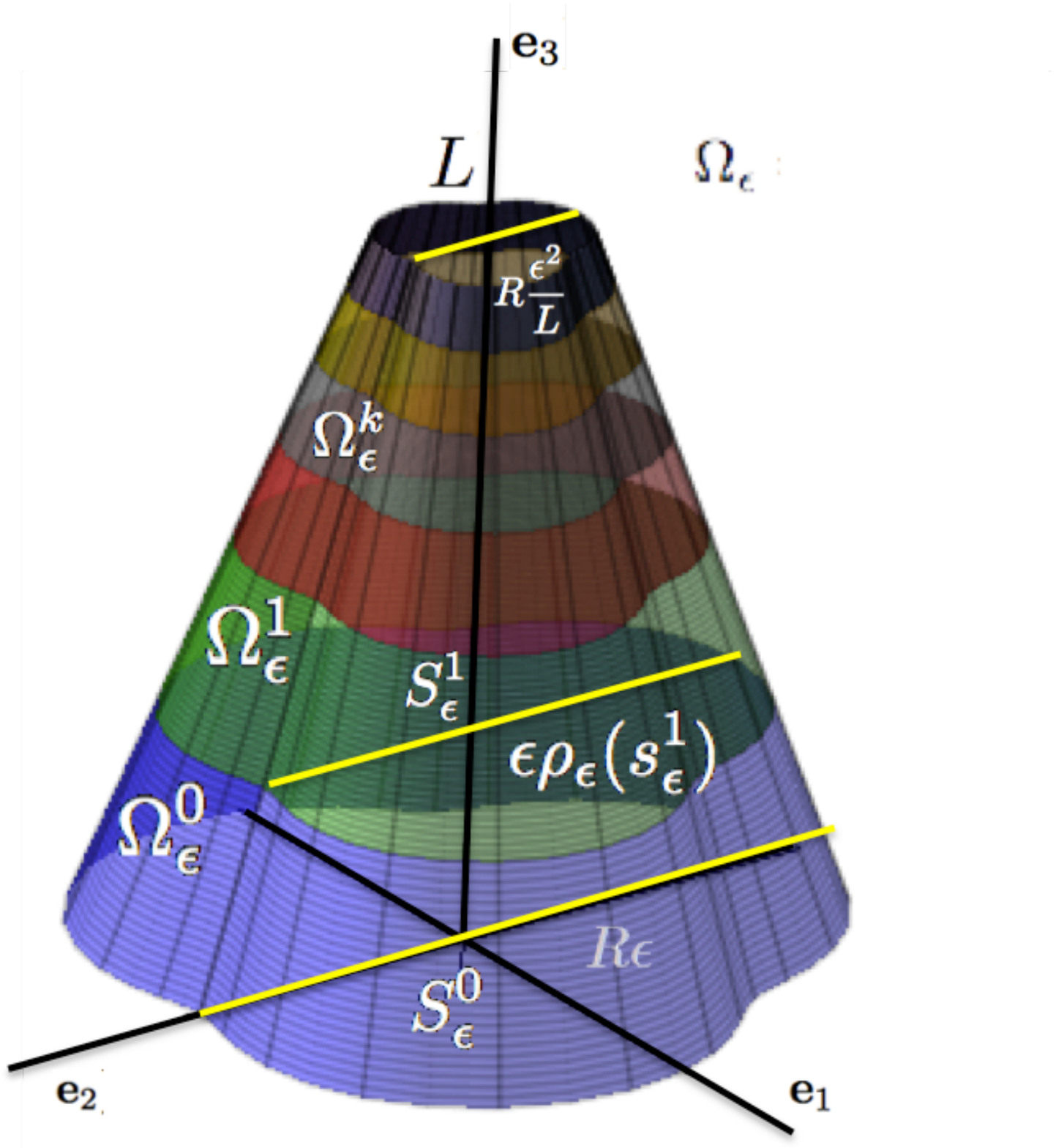}
    \caption{Partition of straight rod $\Omega_\eps$. }
    \label{rod1}
\end{figure}
{\it Step 2.}  Rigid body approximation of $u$ in the portions.
\medskip

Since $\Omega^k_{\eps}$ $(0\leq k \leq N_{\eps} -2)$ is obtained by transforming $\Omega^0_{\eps}$ by a dilation of ratio $\rho_\eps (s^{k}_\eps) $  we can conclude that $\Omega^k_{\eps}$ $(0\leq k \leq N_{\eps} -2)$
is star-shaped with respect to a ball of radius $ R_1\eps \rho_\eps (s^{k+1}_\eps)$ and its diameter is less than $(R+1)\eps\rho_\eps (s^{k}_\eps)$. Moreover, the last portion $\Omega^{N_{\eps}-1}_\eps$ is star-shaped with respect to a ball of radius $ R_1\eps \rho_\eps (s^{N_\eps+1}_\eps)$ and its diameter is less than $2(R+1)\eps\rho_\eps (s^{N_\eps-1}_\eps).$
\smallskip

From Theorem 2.3 in \cite{Gr08} there exists a rigid body displacement $r_k$ $(0\leq k \leq N_\eps-1)$ such that
\begin{equation}\label{rigid}
\begin{aligned}
 \| u - r_k \|^2_{L^2(\Omega^k_{\eps}; \R^3)} & \leq C(R + 1)^2 \eps^2 \rho_\eps( s^k_{\eps})^2 \| (\nabla u)_{\mathcal{S}} \|^2_{[L^2(\Omega^k_{\eps})]^9}, \\
\| \nabla(u - r_k) \|^2_{[L^2(\Omega^k_{\eps}; \R^3)]^9} & \leq C (R+1)^2 \| (\nabla u)_{\mathcal{S}} \|^2_{[L^2(\Omega^k_{\eps})]^9}.
\end{aligned} \end{equation}
The constants depend only on the reference cross-section $\omega$ and on the ratio between the diameter of the portion and the radius of the ball inside (see Theorem 2.3 in \cite{Gr08})
\begin{equation}\label{ratio} 
\ds \frac{(R+1)\eps \rho_\eps(s^k_{\eps})}{R_1\eps \rho_\eps(s^{k+1}_{\eps})} = \frac{R+1}{R_1}\frac{\rho_\eps(s^k_{\eps})}{\ds \rho_\eps(s^{k}_{\eps}) - \frac{\eps\rho_\eps(s^{k}_{\eps})}{L}\Big(1-\frac{\eps}{L}\Big)} = \frac{R+1}{R_1}\frac{1}{\rho_\eps(\eps)}\leq {4\over 3}\frac{R+1}{R_1}, \, 0\leq k \leq N_\eps-2.
\end{equation}
Observe that for the last portion the ratio is less than $\ds 4\frac{R+1}{R_1}.$
\medskip

{\it Step 3.} First estimate in \eqref{estimates}. 
\medskip

Recall  that the rigid body displacements $r_k$ are of the form
$$r_k(x) = A_k + B_k \land (x_1{\Ge}_1 +x_2{\Ge}_2 + (x_3 - s^k_{\eps}){\Ge}_3), \quad x=(x_1, x_2, x_3) \in \Omega^k_{\eps} \textrm{ and } A_k, B_k \in \R^3.$$
Now, we are going to prove $(0\leq k \leq N_\eps-2)$\footnotemark$\!$ \footnotetext{If $k=N_\eps-1$ we have to replace $s^{k+1}_{\eps}$ by $L$.}
\begin{align}
& \| \mathcal{U} - A_k - B_k \land ( x_3 - s_{\eps,k}){\Ge}_3  \|^2_{L^2((s^k_{\eps}, s^{k+1}_{\eps}); \R^3)} \leq  C \| (\nabla u)_{\mathcal{S}} \|^2_{[L^2(\Omega^k_{\eps})]^9}\label{U},\\
& \| \mathcal{R} - B_k \|^2_{L^2((s^k_{\eps}, s^{k+1}_{\eps}); \R^3)} \leq  \frac{C}{\eps^2 \rho_\eps(s^{k+1}_{\eps})^2 } \| (\nabla u)_{\mathcal{S}} \|^2_{[L^2(\Omega^k_{\eps})]^9}\label{R}.
\end{align}
 The constants do not depend on $k$ and $\eps$. 
\smallskip

 The proof is similar for both inequalities, we will show only the first one. Taking the mean value of $u-r_k$ over the cross-sections of the portion $\Omega^k_{\eps}$ and by the definition of the elementary displacement and \eqref{integral}  we have
\begin{eqnarray*}
  \lefteqn{\| \mathcal{U} - A_k - B_k \land ( x_3 - s_{\eps,k}){\Ge}_3  \|^2_{L^2((s^k_{\eps}, s^{k+1}_{\eps}); \R^3)} = \int_{s^k_{\eps}}^ {s^{k+1}_{\eps}} |\mathcal{U}(x_3) - A_k - B_k \land ( x_3 - s_{\eps,k}){\Ge}_3|^2\, dx_3}\\
 & =& \int_{s^k_{\eps}}^ {s^{k+1}_{\eps}} \Big| \frac{1}{|\omega| \rho_\eps(x_3)^2\eps^2}\int_{\omega_{\eps, x_3}} [u(x_1, x_2, x_3)  - r_k(x_1, x_2, x_3)]\, dx_1 dx_2\Big|^2\, dx_3\\
  &\leq& \int_{s^k_{\eps}}^{ s^{k+1}_{\eps}} \frac{1}{|\omega| \rho_\eps(x_3)^2\eps^2}\int_{\omega_{\eps, x_3}}|u(x)  - r_k(x)|^2\, dx\\
  &\leq& \frac{1}{|\omega| \rho_\eps(s^{k+1}_\eps)^2\eps^2} \int_{\Omega^k_{\eps}} |u(x)  - r_k(x)|^2\, dx.
 \end{eqnarray*}
 Then using (\ref{rigid})$_1$ and taking into account (\ref{ratio}) we obtain the expected estimate 
  \begin{align*}
    \| \mathcal{U} - A_k - B_k \land ( x_3 - s_{\eps,k}){\Ge}_3  \|^2_{L^2((s^k_{\eps}, s^{k+1}_{\eps}); \R^3)} &\leq \frac{C(R + 1)^2 \eps^2 \rho_\eps( s^k_{\eps})^2}{|\omega|\eps^2 \rho_\eps(s^{k+1}_\eps)^2} \| (\nabla u)_{\mathcal{S}} \|^2_{[L^2(\Omega^k_{\eps})]^9}
   \leq C \| (\nabla u)_{\mathcal{S}} \|^2_{[L^2(\Omega^k_{\eps})]^9},
  \end{align*}
where the constant does not depend on $\eps$ and $k$. 
\vskip 1mm

Consequently, from (\ref{U}) and (\ref{R}), taking into account  the definition of the elementary displacement and $\ds \int_{\omega_{\eps,x_3}} x_{\alpha}^2 \, dx_1dx_2 = \eps^4 \rho_\eps(x_3)^4 I_{\alpha}$ we have
\begin{align*}
&\|  U_e - r_k  \|_{L^2(\Omega^k_{\eps})} \leq \|  \mathcal{U} - A_k - B_k \land ( x_3 - s_{\eps,k}){\Ge}_3 \|_{L^2(\Omega^k_{\eps})} + \| (\mathcal{R} - B_k)\land (x_1{\Ge}_1 +x_2{\Ge}_2)
\|_{L^2(\Omega^k_{\eps})}\\
&\leq \int_{s^k_{\eps}}^{ s^{k+1}_{\eps}} |\omega| \rho_\eps(x_3)^2\eps^2 |\mathcal{U}(x_3)  - A_k - B_k \land ( x_3 - s_{\eps,k}){\Ge}_3|^2\, dx_3 + C \eps^4 \rho_\eps(s^k_{\eps})^4  \|\mathcal{R} - B_k\|^2_{L^2(s^k_{\eps}, s^{k+1}_{\eps})} \\
& \leq C  \eps^2 \rho_\eps(s^k_{\eps})^2 \| (\nabla u)_{\mathcal{S}} \|^2_{[L^2(\Omega^k_{\eps})]^9}.
\end{align*}

Thus, we can replace $r_k$ by $U_e$ in \eqref{rigid}$_1$ 
 $$
  \| u - U_e \|^2_{L^2(\Omega^k_{\eps}; \R^3)} \leq C \eps^2 \rho_\eps( s^k_{\eps})^2 \| (\nabla u)_{\mathcal{S}} \|^2_{[L^2(\Omega^k_{\eps})]^9}.$$
 Moreover, since $\ds 1\leq \frac{\rho_\eps(s^k_{\eps})}{\rho_\eps(x_3)}\leq 2$ for $x_3\in (s^k_{\eps}, s^{k+1}_{\eps})$, we get
$$
\Big\|\frac{ u - U_e}{\rho_\eps} \Big\|^2_{L^2(\Omega^k_{\eps}; \R^3)} \leq C \eps^2 \| (\nabla u)_{\mathcal{S}} \|^2_{[L^2(\Omega^k_{\eps})]^9}.
$$
Adding all these inequalities lead to the first estimate involving the warping
\begin{equation}\label{estU}
\Big\| \frac{u - U_e}{\rho_\eps} \Big\|^2_{L^2(\Omega_{\eps}; \R^3)}\leq
C \eps^2 \| (\nabla u)_{\mathcal{S}} \|^2_{[L^2(\Omega_{\eps})]^9}.
\end{equation}
\medskip

{\it Step 4.} Second estimate in \eqref{estimates}.
 \medskip
 
First of all, we compute the derivative of $\mathcal{U}$ with respect to $x_3$. Since the diameter of the cross-section depends on $x_3$ we rewrite $\mathcal{U}$ performing a change of variables
$$ \mathcal{U} (x_3) = \frac{1}{|\omega|\eps^2}\int_{\omega_\eps} u(\rho_\eps(x_3) s_1, \rho_\eps(x_3) s_2, x_3) \, ds_1 ds_2, $$
where $\ds \omega_\eps =\eps \omega$.
The derivative is given by
$$\frac{d\mathcal{U}}{d x_3}(x_3) = \frac{1}{|\omega|\eps^2}\int_{\omega_\eps} \Big[\frac{\partial u}{\partial x_1} \rho_\eps' (x_3) s_1 + \frac{\partial u}{\partial x_2}\rho_\eps' (x_3) s_2 + \frac{\partial u}{\partial x_3}\Big]  \, ds_1 ds_2, \quad \textrm{ for a.e } x_3 \in (0, L).$$
Undoing the change of variables we get
$$
\frac{ d\mathcal{U}}{d x_3}(x_3) = \frac{1}{|\omega| \eps^2\rho_\eps(x_3)^2}\int_{\omega_{\eps, x_3}} \Big[\frac{\partial u}{\partial x_1} x_1 \frac{\rho_\eps' (x_3)}{\rho_\eps(x_3)}  + \frac{\partial u}{\partial x_2} x_2 \frac{\rho_\eps' (x_3)}{\rho_\eps(x_3)} + \frac{\partial u}{\partial x_3}\Big]  \, dx_1 dx_2, \quad \textrm{ for a.e } x_3 \in (0, L). $$
From (\ref{rigid}) we have 
$$\Big\| \frac{\partial u}{\partial x_i} - B_k \land {\Ge}_{i} \Big\|^2_{L^2(\Omega^k_{\eps}; \R^3)} \leq C \| (\nabla u)_{\mathcal{S}} \|^2_{[L^2(\Omega^k_{\eps})]^9}, \quad i \in \{1,2,3\}.$$
Moreover, from (\ref{R}) we may replace $B_k$ by $\mathcal{R}$
$$\Big\| \frac{\partial u}{\partial x_i} - \mathcal{R} \land {\Ge}_{i} \Big\|^2_{L^2(\Omega^k_{\eps}; \R^3)} \leq C \| (\nabla u)_{\mathcal{S}} \|^2_{[L^2(\Omega^k_{\eps})]^9}, \quad i \in \{1,2,3\}.$$
Adding all these inequalities we obtain 
\begin{equation}\label{deriv}
\Big \| \frac{\partial u}{\partial x_i} - \mathcal{R} \land {\Ge}_{i} \Big\|^2_{L^2(\Omega_{\eps}; \R^3)} \leq C \| (\nabla u)_{\mathcal{S}} \|^2_{[L^2(\Omega_{\eps})]^9}, \quad i \in \{1,2,3\}.
\end{equation}
Taking into account \eqref{integral} we have  for a.e. $x_3 \in (0, L)$
\begin{equation*}
\begin{aligned}
\frac{ d\mathcal{U}}{d x_3}(x_3) - \mathcal{R}(x_3)\land {\Ge}_{3}  &= \frac{1}{|\omega| \eps^2\rho_\eps(x_3)^2}\int_{\omega_{\eps, x_3}} \Big[ \Big(\frac{\partial u}{\partial x_1} (x)- \mathcal{R}(x_3) \land {\Ge}_{1}\Big) x_1 \frac{\rho_\eps' (x_3)}{\rho_\eps(x_3)}\\
&+ \Big(\frac{\partial u}{\partial x_2}(x) - \mathcal{R}(x_3) \land {\Ge}_{2}\Big) x_2 \frac{\rho_\eps' (x_3)}{\rho_\eps(x_3)}+\Big(\frac{\partial u}{\partial x_3} (x)- \mathcal{R}(x_3) \land {\Ge}_{3}\Big)\Big]  \, dx_1 dx_2. 
\end{aligned}
\end{equation*}
Using \eqref{deriv} leads to $(0\leq k \leq N_\eps-2)$\footnotemark$\!$ \footnotetext{If $k=N_\eps-1$ we have to replace $s^{k+1}_{\eps}$ by $L$.}
\begin{align*}
 \Big\| \frac{d \mathcal{U}}{d x_3} - \mathcal{R} \land {\Ge}_{3} \Big \|^2_{L^2(s^k_{\eps}, s^{k+1}_{\eps})} &\leq  \frac{C}{\rho_\eps(s^{k+1}_{\eps})^2} \| (\nabla u)_{\mathcal{S}} \|^2_{[L^2(\Omega^k_{\eps})]^9} + \frac{C}{\eps^2\rho_\eps(s^{k+1}_{\eps})^2} \| (\nabla u)_{\mathcal{S}} \|^2_{[L^2(\Omega^k_{\eps})]^9} \\
 &\leq \frac{C}{\eps^2\rho_\eps(s^{k+1}_{\eps})^2} \| (\nabla u)_{\mathcal{S}} \|^2_{[L^2(\Omega^k_{\eps})]^9}.
\end{align*}
  Hence, since $\ds 1\leq \frac{\rho_\eps(x_3)}{\rho_\eps(s^{k+1}_{\eps})}\leq 2$ for $x_3\in (s^k_{\eps}, s^{k+1}_{\eps})$ we obtain
    \begin{equation}\label{DU}
 \Big\| \rho_\eps\Big(\frac{d \mathcal{U}}{d x_3} - \mathcal{R} \land {\Ge}_{3}\Big)  \Big\|^2_{L^2(s^k_{\eps}, s^{k+1}_{\eps})}
 \leq \frac{C}{\eps^2} \| (\nabla u)_{\mathcal{S}} \|^2_{[L^2(\Omega_{\eps})]^9}.
 \end{equation}
  Adding all these inequalities we get the desired estimate
  \begin{equation}\label{derU}
\Big \| \rho_\eps\Big(\frac{d \mathcal{U}}{dx_3} - \mathcal{R} \land {\Ge}_{3}\Big)  \Big\|^2_{L^2((0, L); \R^3)}
 \leq \frac{C}{\eps^2} \| (\nabla u)_{\mathcal{S}} \|^2_{[L^2(\Omega_{\eps})]^9}.
 \end{equation}

 \medskip
 
 {\it Step 5.} Third estimate in \eqref{estimates}.
  \medskip
 
First of all, we introduce the function:
 
$$ V(x_3) = \frac{1}{\eps^4\rho_\eps(x_3)^4} \int_{\omega_{\eps, x_3}} [(x_1{\Ge}_1+x_2{\Ge}_2)  \land u(x_1, x_2, x_3)] \, dx_1 dx_2.$$
To calculate the derivative with respect to $x_3$ we perform a change of variables which allows us to write the function $V$ as follows:
$$
\begin{aligned}
 V(x_3) &= \frac{1}{\eps^4\rho_\eps(x_3)^2} \int_{\omega_\eps} (\rho_\eps(x_3)s_1{\Ge}_1+\rho_\eps(x_3)s_2{\Ge}_2) \land u(\rho_\eps(x_3)s_1, \rho_\eps(x_3)s_2, x_3)] \, ds_1 ds_2\\
& = \frac{1}{\eps^4\rho_\eps(x_3)} \int_{\omega_\eps} (s_1{\Ge}_1+s_2{\Ge}_2) \land u(\rho_\eps(x_3)s_1, \rho_\eps(x_3)s_2, x_3)] \, ds_1 ds_2.
\end{aligned}
$$
Then deriving with respect to $x_3$ gives (for a.e. $x_3\in (0,L)$)
$$
\begin{aligned}
\frac{dV}{dx_3}(x_3)= &\frac{-2\rho_\eps'(x_3)}{\eps^4\rho_\eps(x_3)^2}\int_{\omega_\eps} [( s_1{\Ge}_1+s_2{\Ge}_2)  \land u(\rho_\eps(x_3)s_1, \rho_\eps(x_3)s_2, x_3)] \, ds_1 ds_2 \\
+ &\frac{1}{\eps^4\rho_\eps(x_3)}\int_{\omega_\eps}\Big[(s_1{\Ge}_1+s_2{\Ge}_2)\land \Big(\frac{\partial u}{\partial x_1} \rho_\eps' (x_3) s_1 + \frac{\partial u}{\partial x_2}\rho_\eps' (x_3) s_2 + \frac{\partial u}{\partial x_3}\Big)\Big] \, ds_1 ds_2
\end{aligned}
$$
Undoing the change of variables we have (for a.e. $x_3\in (0,L)$)
\begin{equation}\label{derivV}
\begin{aligned}
\frac{dV}{ dx_3} (x_3)= &\frac{-2\rho_\eps'(x_3)}{\eps^4\rho_\eps(x_3)^5}\int_{\omega_\eps,x_3} [(x_1{\Ge}_1+x_2{\Ge}_2)  \land u(x_1, x_2, x_3)] \, dx_1 dx_2 \nonumber\\
+ & \frac{1}{\eps^4\rho_\eps(x_3)^4}\int_{\omega_{\eps, x_3}}\Big[(x_1{\Ge}_1+x_2{\Ge}_2) \land \Big(\frac{\partial u}{\partial x_1} \frac{\rho_\eps'(x_3)}{\rho_\eps(x_3)}x_1 + \frac{\partial u}{\partial x_2}\frac{\rho_\eps'(x_3)}{\rho_\eps(x_3)}x_2+ \frac{\partial u}{\partial x_3}\Big)\Big] \, dx_1 dx_2.
\end{aligned}
\end{equation}
In view of \eqref{integral} we can write (for a.e. $x_3\in (0,L)$)
\begin{equation*}
\begin{aligned}
&\frac{dV}{ dx_3}(x_3) = \frac{-2\rho_\eps'(x_3)}{\eps^4\rho_\eps(x_3)^5}\int_{\omega_\eps,x_3} \Big[(x_1{\Ge}_1+x_2{\Ge}_2)  \land \Big(u(x)- \mathcal{U}(x_3) - \mathcal{R}(x_3) \land (x_1{\Ge}_1+x_2{\Ge}_2)\Big)\Big] \, dx_1 dx_2\\
 &+  \frac{\rho_\eps'(x_3)}{\eps^4\rho_\eps(x_3)^5}\int_{\omega_{\eps, x_3}}\Big\{(x_1{\Ge}_1+x_2{\Ge}_2) \land \Big[ x_1\Big(\frac{\partial u}{\partial x_1} (x) - \mathcal{R}(x_3) \land {\Ge}_1\Big)+x_2\Big(\frac{\partial u}{\partial x_2}(x) - \mathcal{R}(x_3) \land {\Ge}_2)\Big)\Big\}\, dx_1 dx_2\\
&+ \frac{1}{\eps^4\rho_\eps(x_3)^4}\int_{\omega_{\eps, x_3}}(x_1{\Ge}_1+x_2{\Ge}_2) \land\Big(\frac{\partial u}{\partial x_3} (x)- \mathcal{R}(x_3) \land {\Ge}_{3}\Big)\, dx_1 dx_2.
 \end{aligned}
 \end{equation*}
  Using (\ref{estU}) and (\ref{DU}) leads to $(0\leq k \leq N_\eps-2)$\footnotemark$\!$ \footnotetext{If $k=N_\eps-1$ we have to replace $s^{k+1}_{\eps}$ by $L$.}
 \begin{align*}
 \Big \| \frac{dV}{dx_3}  \Big\|^2_{L^2((s^k_{\eps}, s^{k+1}_{\eps}); \R^3)}&\leq  \frac{C}{\eps^2\rho_\eps(s^{k+1}_{\eps})^2}\Big\|\frac{ u - U_e}{\rho_\eps} \Big\|^2_{L^2(\Omega^k_{\eps}; \R^3)}  + \frac{C}{\eps^2\rho_\eps(s^{k+1}_{\eps})^2} \sum_{i=1}^3 \Big\| \frac{\partial  u}{d x_3} - \mathcal{R} \land {\Ge}_{i}  \Big\|^2_{L^2(\Omega^k_{\eps}; \R^3)}\\
  &\leq  \frac{C}{\rho_\eps(s^{k+1}_{\eps})^4} \| (\nabla u)_{\mathcal{S}} \|^2_{[L^2(\Omega^k_{\eps}]^9} + \frac{C}{\eps^4\rho_\eps(s^{k+1}_{\eps})^4} \| (\nabla u)_{\mathcal{S}} \|^2_{[L^2(\Omega^k_{\eps}]^9}\\
  &\leq \frac{C}{\eps^4\rho_\eps(s^{k+1}_{\eps})^4} \| (\nabla u)_{\mathcal{S}} \|^2_{[L^2(\Omega^k_{\eps})]^9}.
   \end{align*}
   Thus, since $\ds 1\leq \frac{\rho_\eps(x_3)}{\rho_\eps(s^{k+1}_{\eps})}\leq 2$ for $x_3\in (s^k_{\eps}, s^{k+1}_{\eps})$ and adding all these inequalities we have
  \begin{equation*}
 \Big \| \rho_\eps^2 \frac{dV}{d x_3} \Big\|^2_{L^2((0, L); \R^3)} \leq \frac{C}{\eps^4}  \| (\nabla u)_{\mathcal{S}} \|^2_{[L^2(\Omega_{\eps})]^9}.
  \end{equation*}

Since $\ds (I_1 + I_2) \mathcal{R} = V + \frac{I_1}{I_2}(V \cdot {\Ge}_1) {\Ge}_1 + \frac{I_2}{I_1}(V \cdot {\Ge}_2) {\Ge}_2$ we get the required estimate
  \begin{equation}\label{derR}
\Big  \| \rho_\eps^2 \frac{d\mathcal{R}}{d x_3} \Big \|^2_{L^2((0, L); \R^3)} \leq \frac{C}{\eps^4} \| (\nabla u)_{\mathcal{S}}  \|^2_{[L^2(\Omega_{\eps})]^9}.
 \end{equation}
 \medskip
 {\it Step 6.} Fourth estimate.
  
  Observe that
\begin{align*}
\frac{\partial}{\partial x_\alpha} ( u - U_e) &= \frac{\partial u}{\partial x_\alpha} - \mathcal{R} \land {\Ge}_\alpha, \quad \textrm{ for } \alpha \in \{1,2\},\\
 \frac{\partial}{\partial x_3} ( u - U_e) &= \frac{\partial u}{\partial x_3} - \frac{\partial \mathcal{U}}{\partial x_3} - \frac{\partial \mathcal{R}}{\partial x_3} \land (x_1{\Ge}_1+x_2{\Ge}_2)\\
  & =\frac{\partial u}{\partial x_3} -  \mathcal{R} \land {\Ge}_3 + \mathcal{R} \land {\Ge}_3 - \frac{\partial \mathcal{U}}{\partial x_3} - \frac{\partial \mathcal{R}}{\partial x_3} \land (x_1{\Ge}_1+x_2{\Ge}_2).
\end{align*}
 From these expressions and taking into account \eqref{deriv}, \eqref{derU} and \eqref{derR} we can conclude 
 $$\| \nabla\bar{u} \|^2_{[L^2(\Omega_{\eps}; \R^3)]^9} \leq C \| (\nabla u)_{\mathcal{S}} \|^2_{[L^2(\Omega_{\eps})]^9},$$
 which ends the proof.
 \end{proof}

 \section{ Estimates for the clamped rod at the bottom.}
 From now on, we will assume that the rod $\Omega_\eps$ is clamped at the bottom, $\Gamma_{\eps,0}=\omega_{\eps,0} \times \{0\}.$ Then the space of admissible displacements of the rod is
 $$H^1_{\Gamma_{\eps,0}}(\Omega_\eps; \R^3) = \{ u \in H^1(\Omega_\eps; \R^3)\; | \; u = 0 \textrm{ on } \Gamma_{\eps,0}\}.$$
 Observe that the elementary displacement $U_e$ associated to any $u \in H^1_{\Gamma_{\eps,0}}(\Omega_\eps; \R^3)$ is equal to zero in the fixed part of the rod, $\mathcal{U}(0) = \mathcal{R}(0)=0$.
 
Using estimates (\ref{estimates}) and the boundary condition we deduce estimates on $\ds \mathcal{R}, \frac{d\mathcal{U}}{dx_3}$ and $\mathcal{U}$.
\begin{lemma}\label{Lem31}
Assuming the rod clamped at the bottom, then we have
\begin{equation} \label{estimates1}
\left\{
\begin{aligned}
 \big \| \rho_\eps\mathcal{R}  \big\|_{L^2((0, L); \R^3)}  &\leq \frac{C L}{\eps^2} \| (\nabla u)_{\mathcal{S}} \|_{[L^2(\Omega_{\eps})]^9},\\
  \Big\| \rho_\eps\, \frac{d \mathcal{U}_\alpha}{d x_3} \Big\|_{L^2(0, L)} &\leq \frac{CL}{\eps^2} \| (\nabla u)_{\mathcal{S}} \|_{[L^2(\Omega_{\eps})]^9},\quad\hbox{for }\alpha\in\{1,2\}, \\
   \Big\| \rho_\eps\, \frac{d \mathcal{U}_3}{d x_3} \Big\|_{L^2(0, L)}  &\leq \frac{C}{\eps} \| (\nabla u)_{\mathcal{S}} \|_{[L^2(\Omega_{\eps})]^9},\\
 \big\| \mathcal{U}_\alpha \|_{L^2(0, L)} & \leq \frac{CL^2}{\eps^2} \| (\nabla u)_{\mathcal{S}} \|_{[L^2(\Omega_{\eps})]^9},  \quad\hbox{for }\alpha\in\{1,2\},\\
 \big\| \mathcal{U}_3 \|_{L^2(0, L)}&\leq \frac{CL}{\eps} \| (\nabla u)_{\mathcal{S}} \|_{[L^2(\Omega_{\eps})]^9},
\end{aligned}
\right.
\end{equation}
 The constants are independent of $\eps$ and $L$.
\end{lemma} 

\begin{proof}
We begin with the proof of the first estimate in \eqref{estimates1}. Since $\mathcal{R}(0)=0$ by integration by parts we have
$$\int_0^L 2 \rho^3_\eps(x_3) \mathcal{R}(x_3) \frac{d \mathcal{R}}{d x_3}(x_3) \,dx_3 = - \int_0^L 3 \rho^2_\eps(x_3)\rho'_\eps(x_3) \mathcal{R}^2(x_3) \,dx_3 + \rho^3_\eps(L) \mathcal{R}^2(L).$$
Then taking into account the facts that $\ds {1\over 2 L}\leq -\rho'_\eps(x_3) = \frac{1}{L}\big(1-\frac{\eps}{L}\big)\leq {1\over L}$ ($\ds 0<\eps < \frac{L}{2}$) and $0\leq \rho^3_\eps(L) \mathcal{R}^2(L)$ we get

$$\int_0^L  \rho^2_\eps(x_3) \mathcal{R}^2(x_3) \,dx_3 \leq {2 L\over 3}\int_0^L  \rho^3_\eps(x_3) \mathcal{R}(x_3) \frac{d \mathcal{R}}{d x_3}(x_3) \,dx_3.$$
Hence, by the Cauchy's inequality it follows that:
$$\int_0^L  \rho^3_\eps(x_3) \mathcal{R}(x_3) \frac{d \mathcal{R}}{d x_3}(x_3) \,dx_3\leq   \| \rho_\eps\mathcal{R}  \big\|_{L^2((0, L); \R^3)}  \Big\| \rho_\eps^2 \frac{d\mathcal{R}}{d x_3}  \Big\|_{L^2((0, L); \R^3)}$$
Finally, the above inequalities together with the third estimate in (\ref{estimates}) allow us to obtain the required estimate
\begin{equation}\label{DU0}
\| \rho_\eps\mathcal{R}  \big\|_{L^2((0, L); \R^3)}  \leq {2L\over 3}\Big\| \rho_\eps^2 \frac{d\mathcal{R}}{d x_3}  \Big\|_{L^2((0, L); \R^3)}\leq \frac{CL}{\eps^2} \| (\nabla u)_{\mathcal{S}} \|_{[L^2(\Omega_{\eps})]^9}.
\end{equation}
 The constant is independent of $\eps$  and $L$.
 \vskip 1mm
The second estimate follows from \eqref{estimates}$_2$ and \eqref{DU0}:
\begin{equation}\label{DU1}
\Big\| \rho_\eps\, \frac{d \mathcal{U}}{d x_3} \Big\|_{L^2((0, L); \R^3)} \leq  \Big\| \rho_\eps\, \Big(\frac{d \mathcal{U}}{d x_3} - \mathcal{R} \land e_{3}\Big)  \Big\|_{L^2((0, L))} + \|\rho_\eps  (\mathcal{R} \land e_{3})\| _{L^2((0, L); \R^3)}  \leq \frac{CL}{\eps^2} \| (\nabla u)_{\mathcal{S}} \|_{[L^2(\Omega_{\eps})]^9}.
\end{equation}
From the second estimate in \eqref{estimates} we obtain a better estimate for $\ds \Big\| \rho_\eps\, \frac{d \mathcal{U}_3}{d x_3} \Big\|_{L^2(0, L)}$
\begin{equation}\label{DU2}
\Big\| \rho_\eps\, \frac{d \mathcal{U}_3}{d x_3} \Big\|_{L^2(0, L)} =  \Big\| \rho_\eps\, \Big(\frac{d \mathcal{U}}{d x_3} - \mathcal{R} \land \Ge_{3}\Big) \cdot \Ge_{3} \Big\|_{L^2((0, L); \R^3)} \leq \frac{C}{\eps} \| (\nabla u)_{\mathcal{S}} \|_{[L^2(\Omega_{\eps})]^9}.
\end{equation}

Finally, the estimates for $\mathcal{U}$ follows by a similar computation to  $\mathcal{R}$. We first prove
$$\| \mathcal{U}^\eps_i \|_{L^2(0,L)} \leq  2L\Big\| \rho_\eps\, \frac{d \mathcal{U}^\eps_i}{d x_3} \Big\|_{L^2(0, L)}\qquad \textrm{ for } i =1,2,3$$
then due to \eqref{estimates1}$_2$-\eqref{estimates1}$_3$ we get \eqref{estimates1}$_4$-\eqref{estimates1}$_5$.
%
\end{proof}

In view of the definition of the elementary displacement \eqref{Decomp} we can write explicitly the components of the displacement, the gradient and the symmetric gradient of the displacement

\begin{equation}\label{comp disp}
\left\{
\begin{aligned}
u_1(x) &= \mathcal{U}_1(x_3) - x_2 \mathcal{R}_3(x_3) + \bar{u}_1(x),\\
u_2(x) &= \mathcal{U}_2(x_3) + x_1 \mathcal{R}_3(x_3) + \bar{u}_2(x),\\
u_3(x) &= \mathcal{U}_3(x_3) - x_1 \mathcal{R}_2(x_3) +x_2 \mathcal{R}_1(x_3)+ \bar{u}_3(x).
\end{aligned}
\right. 
\end{equation}

\begin{remark}
Notice that, due to the definition of $\mathcal{U}$, $\mathcal{R}$ and \eqref{integral} we know that the warping satisfies 
\begin{equation}\label{warping cond}
\int_{\omega_{\eps,x_3}} \bar{u}_i \, dx_1 dx_2 = \int_{\omega_{\eps,x_3}} (x_1\bar{u}_2-x_2\bar{u}_1) \, dx_1 dx_2=\int_{\omega_{\eps,x_3}} x_\alpha\bar{u}_3 \, dx_1 dx_2=0, \quad i\in\{1, 2, 3\},\;\; \alpha\in\{1,2\}. 
\end{equation}
\end{remark}

\begin{equation} \label{comp grad}
\nabla u=  \left(\begin{array}{ccc} \ds \frac{\partial\bar{u}_1}{\partial x_1} & \ds - \mathcal{R}_3 + \frac{\partial\bar{u}_1}{dx_2} & \ds \frac{d\mathcal{U}_1}{dx_3} - x_2\frac{d \mathcal{R}_3}{dx_3}+\frac{\partial\bar{u}_1}{\partial x_3} \\
[6mm] \ds  \mathcal{R}_3 + \frac{\partial\bar{u}_2}{\partial x_1} & \ds \frac{\partial\bar{u}_2}{\partial x_2} & \ds \frac{d\mathcal{U}_2}{dx_3} + x_1\frac{d \mathcal{R}_3}{dx_3}+\frac{\partial\bar{u}_2}{\partial x_3} \\
[6mm] \ds -\mathcal{R}_2 + \frac{\partial\bar{u}_3}{\partial x_1} & \ds \mathcal{R}_1 + \frac{\partial\bar{u}_3}{\partial x_2} & \ds \frac{d\mathcal{U}_3}{dx_3} - x_1\frac{d \mathcal{R}_2}{dx_3} + x_2\frac{d \mathcal{R}_1}{dx_3} +\frac{\partial\bar{u}_3}{\partial x_3}\end{array}\right)
\end{equation}

\medskip

\begin{equation} \label{sym grad}
(\nabla u)_{\mathcal{S}}=  \left(\begin{array}{ccc} \ds \frac{\partial\bar{u}_1}{\partial x_1} & \ds  \frac{1}{2} \Big(\frac{\partial\bar{u}_1}{\partial x_2} +\frac{\partial\bar{u}_2}{\partial x_1}\Big)  & \ds \frac{1}{2}\Big(\frac{d\mathcal{U}_1}{dx_3} - x_2\frac{d \mathcal{R}_3}{dx_3}-\mathcal{R}_2 + \frac{\partial\bar{u}_3}{\partial x_1} +\frac{\partial\bar{u}_1}{\partial x_3}\Big) \\
[6mm]  * & \ds \frac{\partial\bar{u}_2}{\partial x_2} & \ds \frac{1}{2}\Big(\frac{d\mathcal{U}_2}{dx_3} + x_1\frac{d \mathcal{R}_3}{dx_3} + \mathcal{R}_1 + \frac{\partial\bar{u}_3}{\partial x_2} +\frac{\partial\bar{u}_2}{\partial x_3}\Big) \\[6mm] * & *  & \ds \frac{d\mathcal{U}_3}{dx_3} - x_1\frac{d \mathcal{R}_2}{dx_3}+x_2\frac{d \mathcal{R}_1}{dx_3} + \frac{\partial\bar{u}_3}{\partial x_3}\end{array}\right)
\end{equation}

The previous Lemma \ref{Lem31} allows us to established the Korn's inequality for any displacement $u \in H^1_{\Gamma_{\eps,0}}(\Omega_\eps; \R^3).$

\begin{lemma}

Assuming   the rod clamped at the bottom boundary, then we have
\begin{equation}\label{Korn inequality}
\left\{
\begin{aligned}
\| \nabla u \|_{[L^2(\Omega_{\eps}; \R^3)]^9} &\leq C\frac{L}{\eps} \| (\nabla u)_{\mathcal{S}} \|_{[L^2(\Omega_{\eps})]^9},\\
\Big\| \frac{u_\alpha}{\rho_\eps} \Big\|_{L^2(\Omega_{\eps})}  & \leq C \frac{L^2}{\eps}  \| (\nabla u)_{\mathcal{S}} \|_{[L^2(\Omega_{\eps})]^9},\quad\hbox{for } \alpha\in\{1,2\},\\
\Big\| \frac{u_3}{\rho_\eps} \Big\|_{L^2(\Omega_{\eps})} &\leq  C L \| (\nabla u)_{\mathcal{S}} \|_{[L^2(\Omega_{\eps})]^9},
\end{aligned}
\right. 
\end{equation}
The constant does not depend on $\eps$ and $L$.
\end{lemma}
\begin{proof}
Recall that any displacement $u\in H^1_{\Gamma_{\eps,0}}(\Omega_\eps; \R^3)$ can be written as $ u = U_e + \bar{u}$. Then we get
$$\| \nabla u \|_{[L^2(\Omega_{\eps}; \R^3)]^9}  \leq  \| \nabla U_e \|_{[L^2(\Omega_{\eps}; \R^3)]^9} +  \| \nabla \bar{u} \|_{[L^2(\Omega_{\eps}; \R^3)]^9} .$$ 
Using \eqref{comp grad}, \eqref{estimates} and \eqref{estimates1} one has the following estimate:
\begin{equation*}
\begin{aligned}
&\|  \nabla U_e \|_{[L^2(\Omega_{\eps}; \R^3)]^9}\leq  \Big\| \eps \rho_\eps \frac{d \mathcal{U}}{d x_3} \Big\|_{L^2((0, L); \R^3)} + \big\| \eps \rho_\eps\mathcal{R} \land (\Ge_1+\Ge_2) \big\|_{L^2((0, L); \R^3)}\\
&+\Big\| \frac{d\mathcal{R}}{d x_3} \land (x_1\Ge_1+x_2\Ge_2) \Big\|_{L^2(\Omega_{\eps}; \R^3)}
\leq \frac{CL}{\eps} \| (\nabla u)_{\mathcal{S}} \|_{[L^2(\Omega_{\eps})]^9} +  \frac{CL}{\eps} \| (\nabla u)_{\mathcal{S}} \|_{[L^2(\Omega_{\eps})]^9} +   C\| (\nabla u)_{\mathcal{S}} \|_{[L^2(\Omega_{\eps})]^9}\\
& \leq\frac{CL}{\eps} \| (\nabla u)_{\mathcal{S}} \|_{[L^2(\Omega_{\eps})]^9}.
\end{aligned}
\end{equation*} Recall that $\|  \nabla \bar{u} \|_{[L^2(\Omega_{\eps}; \R^3)]^9} \leq C \| (\nabla u)_{\mathcal{S}} \|_{[L^2(\Omega_{\eps}]^9}$. Consequently, we obtain the first estimate in \eqref{Korn inequality}.
\medskip

In view of \eqref{comp disp} and taking into account estimates \eqref{estimates} and \eqref{estimates1} we obtain
\begin{align*}
\Big\| \frac{u_{\alpha}}{\rho_\eps} \Big\|_{L^2(\Omega_{\eps}; \R^3)} &\leq  \| \eps\, \mathcal{U}_{\alpha} \|_{L^2(0, L)} +  \Big\|\frac{x_{3-\alpha} \mathcal{R}_3}{\rho_\eps}\Big\|_{L^2(\Omega_{\eps})} + \Big\| \frac{\bar{u}_\alpha}{\rho_\eps} \Big\|_{L^2(\Omega_{\eps})}\\
& \leq  \frac{CL^2}{\eps} \| (\nabla u)_{\mathcal{S}} \|_{[L^2(\Omega_{\eps})]^9} +  C L \| (\nabla u)_{\mathcal{S}} \|_{[L^2(\Omega_{\eps})]^9} +  C \eps \| (\nabla u)_{\mathcal{S}} \|_{[L^2(\Omega_{\eps})]^9}\\ 
&\leq  \frac{CL^2}{\eps} \| (\nabla u)_{\mathcal{S}} \|_{[L^2(\Omega_{\eps})]^9},  \quad \textrm{ for } \alpha=1,2.
\end{align*}
\begin{align*}
\Big\| \frac{u_{3}}{\rho_\eps} \Big\|_{L^2(\Omega_{\eps}; \R^3)} &\leq  \| \eps\, \mathcal{U}_{3} \|_{L^2(0, L)} + \Big\|\frac{x_{1} \mathcal{R}_2}{\rho_\eps}\Big\|_{L^2(\Omega_{\eps})} + \Big\|\frac{x_{2} \mathcal{R}_1}{\rho_\eps}\Big\|_{L^2(\Omega_{\eps})} + \Big\| \frac{\bar{u}_3}{\rho_\eps} \Big\|_{L^2(\Omega_{\eps})}\\
&\leq  CL \| (\nabla u)_{\mathcal{S}} \|_{[L^2(\Omega_{\eps}]^9},
\end{align*}
which ends the proof.
\end{proof}

\subsection{Rescaling of the rod}

In this paragraph we define an operator which changes the scale. It allows us to transform the rod $\Omega_\eps$ into a domain independent of $\eps.$

Set $\Omega = \omega \times (0, L)$, the reference beam. We rescale $\Omega_\eps$ using the following operator:
$$(\Pi_\eps \phi)(X_1, X_2, x_3) = \phi (\eps \rho_\eps(x_3) X_1, \eps \rho_\eps(x_3) X_2, x_3), \textrm{ for a.e.  } (X_1, X_2, x_3) \in \Omega,$$
defined for any function $\phi$ measurable on $\Omega_\eps$. 
\smallskip

Observe that, if $\phi \in L^2(\Omega_\eps)$ then $(\Pi_\eps \phi) \in L^2(\Omega)$ and we have
\begin{equation}\label{norms}
 \big\| \Pi_\eps \phi \big\|_{L^2(\Omega)} =  \frac{1}{\eps}\Big\| \frac{\phi}{\rho_\eps} \Big\|_{L^2(\Omega_\eps)}.
 \end{equation}
Therefore, taking into account this above relation, we get the estimate for the rescaled warping $\Pi_\eps \bar{u}$
\begin{equation}\label{resc warp}
\big\| \Pi_\eps \bar{u} \big\|_{L^2(\Omega; \R^3)} = \frac{1}{\eps}\Big\| \frac{\bar{u}}{\rho_\eps} \Big\|_{L^2(\Omega_\eps; \R^3)} \leq  C \| (\nabla u)_{\mathcal{S}} \|_{[L^2(\Omega_{\eps})]^9}.
\end{equation}
In order to obtain the estimates for the derivatives of the warping observe that for any  $\phi \in H^1(\Omega_\eps)$
\begin{align}\label{der resc}
  \frac{\partial(\Pi_\eps \phi)}{\partial X_\alpha} &= \eps \rho_\eps \Pi_\eps \Big(\frac{\partial \phi}{\partial x_\alpha} \Big), \textrm{ for } \alpha=1,2,\nonumber\\
 \frac{\partial(\Pi_\eps \phi)}{\partial x_3} &= \eps \rho'_\eps X_1 \Pi_\eps \Big(\frac{\partial \phi}{\partial x_1}\Big) +  \eps \rho'_\eps X_2 \Pi_\eps \Big(\frac{\partial \phi}{\partial x_2}\Big) + \Pi_\eps\Big( \frac{\partial \phi}{\partial x_3}\Big).
 \end{align}
We recall that $\ds ||\rho'_\eps||_{L^\infty(0,L)}\leq {1\over L}$, then from \eqref{estimates} and \eqref{norms} we get
\begin{align} 
  \Big\|\frac{\partial(\Pi_\eps \bar{u})}{\partial X_\alpha} \Big\|_{L^2(\Omega; \R^3)} &= \Big\| \frac{\partial \bar{u}}{\partial x_\alpha} \Big\|_{L^2(\Omega_{\eps}; \R^3)} \leq C \| (\nabla u)_{\mathcal{S}} \|_{[L^2(\Omega_{\eps})]^9},  \textrm{ for } \alpha=1,2,\label{resc derwarp}\\
 \Big\|\rho_\eps \frac{\partial(\Pi_\eps \bar{u})}{\partial x_3} \Big\|_{L^2(\Omega; \R^3)} & \leq {C\over L} \Big(\Big\| \frac{\partial \bar{u}}{\partial x_1} \Big\|_{L^2(\Omega_{\eps}; \R^3)}  + \Big\| \frac{\partial \bar{u}}{\partial x_2} \Big\|_{L^2(\Omega_{\eps}; \R^3)}  \Big) + \frac{1}{\eps}\Big\| \frac{\partial \bar{u}}{\partial x_3} \Big\|_{L^2(\Omega_{\eps}; \R^3)}\nonumber\\
 & \leq   \frac{C}{\eps} \| (\nabla u)_{\mathcal{S}} \|_{[L^2(\Omega_{\eps})]^9}\label{resc derwarp3}.
 \end{align}
 In the same way, all the estimates in the previous sections over $\Omega_\eps$ can be easily transposed over $\Omega.$
 
 \section{Asymptotic behavior of a sequence of displacements }
  Now we consider a sequence of admissible displacements $\{u^\eps\}_\eps$, where $u^\eps\in H^1_{\Gamma_{\eps,0}}(\Omega_\eps; \R^3)$, satisfying
$$
\|(\nabla u^\eps)_S\|_{[L^2(\Omega_{\eps})]^9}\le C\eps^2,
$$
the constant does not depend on $\eps$.
\smallskip

 We are interested to describe the behaviour of the sequence $\{u^\eps\}_\eps$ as $\eps \to 0.$ In the following proposition we introduce the weak limits of the fields of the displacement's decomposition in the rod. We denote
 by $ \ds \rho(x_3)= 1- \frac{x_3}{L}$, $x_3\in [0,L]$,  the strong limit  in $L^\infty(0,L)$ of $\rho_\eps$. Observe that
 \begin{equation}\label{rho}
 0\leq \rho(t) \leq \rho_\eps(t) \quad \textrm{ for } t \in [0,L].
 \end{equation}

First of all, we introduce certain weighted Lebesgue and Sobolev spaces defined in the interval $(0,L)$.
\begin{itemize}
\item $L^2_{\rho^k}(0,L)$, $k\in \N$, consists of locally summable functions $\varphi: (0, L) \to \R$ equipped with the following norm:
$$\| \varphi \|_{L^2_{\rho^k}(0,L)}= \Big(\int_{0}^{L} [\rho^k(t)\varphi(t)]^2 \, dt\Big)^{1/2}.$$
Obseve that, there exists a linear homeomorphism of $L^2(0,L)$ onto $L^2_{\rho^k}(0,L)$
$$T(\psi) = \frac{\psi}{\rho^k}, \quad \textrm{ for } \psi \in L^2(0,L).$$
Then $L^2_{\rho^k}(0,L)=\Bigl\{ \varphi\in L^2_{loc}(0,L)\;\; |\;\; \rho^k \varphi\in L^2(0,L)\Big\}$ endowed with the norm above is a Banach space.
\begin{remark}
Observe that if $\{\Phi_\eps\}_\eps$ is a sequence of functions belonging to $L^2(\Omega)$ and satisfying $\rho^k_\eps\Phi_\eps\rightharpoonup \Psi$ weakly in $L^2(\Omega)$ then $\Phi_\eps\rightharpoonup \Phi=\ds {\Psi\over \rho^k}$ weakly in $L^2_{\rho^k}(\Omega)$. Conversely, if $\{\Phi_\eps\}_\eps$ is a sequence of functions such that $\rho^k_\eps\Phi_\eps$ is uniformly bounded in $L^2(\Omega)$ and satisfies $\Phi_\eps\rightharpoonup \Phi$ weakly in $L^2_{\rho^k}(\Omega)$ then $\rho^k_\eps\Phi_\eps\rightharpoonup \rho^k\Phi$ weakly in $L^2(\Omega)$. Here $k$ belongs to $\N$.
\end{remark}
\item
We define the space $H^1_\rho(0,L)$ as follows:
$$H^1_\rho(0,L)= \Bigl\{ \varphi\in H^1_{loc}(0,L)\;\; |\;\; \rho \varphi' \in L^2(0,L), \varphi \in L^2(0, L) \textrm{ and } \varphi(0)=0\Big\},$$
endowed with the following norm:
$$\| \varphi \|_{H^1_\rho(0,L)}= \Big(\int_{0}^{L} [\rho(t)\varphi'(t)]^2 \, dt\Big)^{1/2}.$$
We use this norm since as in the proof of Lemma \ref{Lem31} we can easily obtain
\begin{equation}\label{Eq 42}
\| \varphi \|_{L^2{(0,L)}}  \leq 2L \| \rho \varphi' \|_{L^2{(0,L)}}, \textrm{ for } \varphi \in H^1_\rho(0,L). 
\end{equation}
Since $\rho^{-k}$, $k \in \mathbb{N}$, is locally integrable we can conclude that $H^1_\rho(0,L)$ is a Banach space, see \cite{K}.



\item Analogously, $H^1_{\rho^2}(0,L)$ and $H^2_{\rho^2}(0,L)$  are the Banach spaces which contain the functions $\varphi: (0, L) \to \R$ such that
$$H^1_{\rho^2}(0,L)= \Bigl\{ \varphi\in H^1_{loc}(0,L)\;\; |\;\; \rho^2 \varphi' \in L^2(0,L), \rho \varphi \in L^2(0, L) \textrm{ and } \varphi(0)=0\Big\},$$
$$H^2_{\rho^2}(0,L)= \Bigl\{ \varphi\in H^2_{loc}(0,L)\;\; |\;\; \rho^2 \varphi'' \in L^2(0,L), \rho \varphi' \in L^2(0, L), \varphi \in L^2(0, L) \textrm{ and } \varphi(0)=0\Big\}.$$
We define their norms to be 
$$\| \varphi \|_{H^1_{\rho^2}(0,L)}= \Big(\int_{0}^{L} [\rho^2(t)\varphi'(t)]^2 \, dt\Big)^{1/2}.$$
$$\| \varphi \|_{H^2_{\rho^2}(0,L)}= \Big(\int_{0}^{L} [\rho^2(t)\varphi''(t)]^2 \, dt\Big)^{1/2}.$$
We can easily prove that
\begin{equation}\label{Eq 43}
\begin{aligned}
& \|\rho\varphi\|_{L^2(0,L)}\le {2L\over 3}\| \varphi \|_{H^1_{\rho^2}(0,L)}\qquad \hbox{ for any }\; \varphi\in H^1_{\rho^2}(0,L),\\
 &\|\rho\varphi'\|_{L^2(0,L)}\le {2L\over 3}\| \varphi \|_{H^2_{\rho^2}(0,L)}\qquad \hbox{ for any }\; \varphi\in H^2_{\rho^2}(0,L),
\end{aligned}
\end{equation}
then \eqref{Eq 42} yields $\ds \|\varphi\|_{L^2(0,L)}\le {4L^2\over 3}\| \varphi \|_{H^2_{\rho^2}(0,L)}$ for any 
$\varphi\in H^2_{\rho^2}(0,L)$.
\end{itemize}

Similarly we define some weighted spaces in the fixed domain $\Omega$ 
$$L^2_\rho(\Omega) =\Big\{ \phi \in L^2_{loc}(\Omega) \;|\; \rho \phi \in L^2(\Omega) \Big\},$$
$$H^1_\rho(\Omega) =\Big\{ \phi \in H^1_{loc}(\Omega) \;|\; \rho \frac{\partial \phi}{\partial x_3} \in L^2(\Omega) \textrm{ and }  \frac{\partial \phi}{\partial X_1} ,\frac{\partial \phi}{\partial X_2},\phi \in L^2(\Omega)\Big\}.$$
  They are Banach spaces endowed with their respective norms
 $$\| \phi \|_{L^2_\rho(\Omega)} = \Big( \int_\Omega [\rho(x_3) \phi(x)]^2 \, dX_1dX_2dx_3\Big)^{1/2}.$$
$$\| \phi \|_{H^1_\rho(\Omega)} = \Big(\int_\Omega \Big\{\Big(\rho \frac{\partial \phi}{\partial x_3}\Big)^2 +\Big(\frac{\partial \phi}{\partial X_2}\Big)^2 +\Big(\frac{\partial \phi}{\partial X_1}\Big)^2 + \phi^2   \Big\}\, dX_1dX_2dx_3\Big)^{1/2}.$$

  \begin{proposition}
 Let $\{u^\eps\}_\eps$ be a sequence of displacements such that $u^\eps \in H^1_{\Gamma_{\eps,0}}(\Omega_\eps; \R^3)$ and
  \begin{equation}\label{assump}
  \|(\nabla u^\eps)_S\|_{[L^2(\Omega_{\eps})]^9}\le C\eps^2,
  \end{equation}
 where the constant $C$ is independent of $\eps$. Then for a subsequence, still denoted by $\{\eps\}$,
 \begin{itemize}
 \item there exist $\mathcal{U} \in [H^1_\rho(0,L)]^3$, $\mathcal{R} \in [H^1_{\rho^2}(0,L)]^3$ and $ \mathcal{Z} \in L^2_\rho((0,L);\R^3)$ such that,  \begin{align} 
& \mathcal{U}^\eps_{\alpha} \rightharpoonup \mathcal{U}_\alpha \textrm{ weakly in } H^1_\rho(0,L), \textrm{ for } \alpha =1,2,\label{weak1}\\
& \frac{1}{\eps} \mathcal{U}^\eps_{3} \rightharpoonup \mathcal{U}_3 \textrm{ weakly in } H^1_\rho(0,L),\label{weak2}\\
&\mathcal{R}^\eps \rightharpoonup \mathcal{R} \textrm{ weakly in } [H^1_{\rho^2}(0,L)]^3,\label{weak3}\\
& \frac{1}{\eps}\Big({d \mathcal{U}^\eps\over dx_3}- \mathcal{R}^\eps\land\Ge_3\Big)\rightharpoonup \mathcal{Z} \textrm{ weakly in } L^2_\rho((0,L);\R^3),\label{weak30}\\
&\mathcal{R}(0)=0,\;\;\mathcal{U}_\alpha(0)=0,\;\; \mathcal{U}_3(0)=0.\label{weak31}
\end{align}
  
\item there exist $\bar{u} \in L^2\big((0, L); H^1(\omega;\R^3)\big)$, $u \in [H^1_\rho(\Omega)]^3$ and $\mathcal{K} \in  H^1_\rho(\Omega;\R^3)$ such that
\begin{align}
\frac{1}{\eps^2} \Pi_\eps( \bar{u}^\eps) &\rightharpoonup \bar{u}  \textrm{ weakly in } L^2\big((0, L); H^1(\omega;\R^3)\big)\label{weak4}\\
\Pi_\eps (u^\eps_\alpha)  &\rightharpoonup u_\alpha \textrm{ weakly in } H^1_\rho(\Omega),\label{weak5}\\
\frac{1}{\eps}\Pi_\eps (u^\eps_3)  &\rightharpoonup u_3 \textrm{ weakly in } H^1_\rho(\Omega),\label{weak6}\\
\frac{1}{\eps}\Pi_\eps (u^\eps-\mathcal{U}^\eps)  &\rightharpoonup \mathcal{K} \textrm{ weakly in } H^1_\rho(\Omega;\R^3).\label{weak7} 
\end{align} 
  Moreover, we have the following relations between the limit fields:
\begin{equation}\label{relation}
\frac{d \mathcal{U}_1}{d x_3} = \mathcal{R}_2, \quad \frac{d \mathcal{U}_2}{d x_3} = - \mathcal{R}_1,
\end{equation}
$$u_1(X_1, X_2, x_3) = \mathcal{U}_1(x_3), \textrm{ for a. e.  } (X_1, X_2, x_3) \in \Omega,$$
$$u_2(X_1, X_2, x_3) = \mathcal{U}_2(x_3), \textrm{ for a. e. } (X_1, X_2, x_3) \in \Omega,$$
\begin{equation}\label{u3}
u_3(X_1, X_2, x_3) = \mathcal{U}_{3}(x_3) - \rho(x_3) X_1{d\mathcal{U}_1\over dx_3}(x_3)- \rho(x_3) X_2{d\mathcal{U}_2\over dx_3}(x_3), \textrm{ for a. e.  } (X_1, X_2, x_3) \in \Omega.
\end{equation}
$$\mathcal{K}_1(X_1, X_2, x_3) = -\rho(x_3)X_2\mathcal{R}_3(x_3) \textrm{ for a.e.  } (X_1, X_2, x_3) \in \Omega.$$
$$\mathcal{K}_2(X_1, X_2, x_3) =  \rho(x_3)X_1\mathcal{R}_3(x_3) \textrm{ for a.e.  } (X_1, X_2, x_3) \in \Omega.$$
$$\mathcal{K}_3(X_1, X_2, x_3) = -\rho(x_3) X_1\frac{d \mathcal{U}_1}{d x_3}(x_3)-\rho(x_3) X_2\frac{d \mathcal{U}_2}{d x_3}(x_3)\textrm{ for a.e. } (X_1, X_2, x_3) \in \Omega.$$

  \end{itemize}
 
 \end{proposition}
 
 \begin{proof} 
 \medskip
 
 First we get the weak limits, up to a subsequence still denoted by $\eps$, of the different fields. Then we derive a few relations between some of them.
 \medskip
 
\noindent{\it Step 1.} The convergences.
\smallskip

\noindent Taking into account \eqref{rho}-\eqref{Eq 42} and \eqref{estimates1}$_2$-\eqref{estimates1}$_3$ we have
$$\| \mathcal{U}^\eps_\alpha \|_{H^1_\rho(0,L)} \leq  C,\qquad \| \mathcal{U}^\eps_3 \|_{H^1_\rho(0,L)} \leq  C\eps,\qquad  \textrm{ for } \alpha=1,2.$$
Then  we obtain the following convergences:
$$ \mathcal{U}^\eps_{\alpha} \rightharpoonup \mathcal{U}_\alpha \textrm{ weakly in } H^1_\rho(0,L), \textrm{ for } \alpha =1,2.$$
$$\frac{1}{\eps} \mathcal{U}^\eps_{3} \rightharpoonup \mathcal{U}_3 \textrm{ weakly in } H^1_\rho(0,L).$$
According to \eqref{rho}  we get 
$$\| \mathcal{R}^\eps_i \|_{H^1_{\rho^2}(0,L)} \leq  \Big\| \rho^2_\eps\, \frac{d \mathcal{R}_i}{d x_3} \Big\|_{L^2(0, L)} \textrm{ for } i =1,2,3.$$
Due to estimates \eqref{estimates}$_3$ and  \eqref{Eq 43} we obtain 
$$\mathcal{R}^\eps \rightharpoonup \mathcal{R} \textrm{ weakly in } [H^1_{\rho^2}(0,L)]^3.$$
Again due to \eqref{rho} we have
$$\Big\| {d \mathcal{U}^\eps\over dx_3}- \mathcal{R}^\eps\land\Ge_3 \Big\|_{L^2_{\rho}((0,L);\R^3)} \leq \Big\| \rho_\eps\, \Big(\frac{d \mathcal{U}}{d x_3} - \mathcal{R} \land e_{3}\Big)  \Big\|_{L^2((0, L); \R^3)}.$$
In view of estimate \eqref{estimates}$_2$ we get 
$$\frac{1}{\eps}\Big({d \mathcal{U}^\eps\over dx_3}- \mathcal{R}^\eps\land\Ge_3\Big)\rightharpoonup \mathcal{Z} \textrm{ weakly in } L^2_\rho((0,L);\R^3).$$
Thanks to the estimates  \eqref{resc warp} and \eqref{resc derwarp} the sequence $\ds \frac{1}{\eps^2} \Pi_\eps( \bar{u}^\eps)$ is bounded in $L^2\big((0, L); H^1(\omega;\R^3)\big)$. Then we obtain
$$\frac{1}{\eps^2} \Pi_\eps( \bar{u}^\eps) \rightharpoonup \bar{u}  \textrm{ weakly in } L^2\big((0, L); H^1(\omega;\R^3)\big).$$
From property \eqref{norms} of the rescaling operator   and the estimate \eqref{Korn inequality}$_2$ we have
\begin{equation}\label{resc disp}
\big\| \Pi_\eps u^\eps_\alpha \big\|_{L^2(\Omega)} =  \frac{1}{\eps}\Big\| \frac{u^\eps_\alpha}{\rho_\eps} \Big\|_{L^2(\Omega_\eps)} \leq C\frac{L}{\eps^2} \| (\nabla u)_{\mathcal{S}} \|_{[L^2(\Omega_{\eps})]^9}, \quad \hbox{for } \alpha=1,2.
\end{equation}
Moreover, taking into account the derivation rule \eqref{der resc} and the estimates \eqref{Korn inequality}$_1$ we have
\begin{align}
  \Big\|\frac{\partial(\Pi_\eps u^\eps_\alpha)}{\partial X_\beta} \Big\|_{L^2(\Omega)} &= \Big\| \frac{\partial u^\eps_\alpha}{\partial x_\beta} \Big\|_{L^2(\Omega_{\eps})} \leq \frac{C}{\eps} \| (\nabla u^\eps)_{\mathcal{S}} \|_{[L^2(\Omega_{\eps})]^9},  \textrm{ for } \alpha, \beta=1,2,\label{resc deru}\\
    \Big\|\rho \frac{\partial(\Pi_\eps u^\eps_\alpha)}{\partial x_3} \Big\|_{L^2(\Omega)} \leq  \Big\|\rho_\eps \frac{\partial(\Pi_\eps u^\eps_\alpha)}{\partial x_3} \Big\|_{L^2(\Omega)} & \leq {C\over L} \Big(\Big\| \frac{\partial u^\eps_\alpha}{\partial x_1} \Big\|_{L^2(\Omega_{\eps})}  + \Big\| \frac{\partial u^\eps_\alpha}{\partial x_2} \Big\|_{L^2(\Omega_{\eps})}  \Big) + \frac{1}{\eps}\Big\| \frac{\partial u^\eps_\alpha}{\partial x_3} \Big\|_{L^2(\Omega_{\eps})} \nonumber\\
   & \leq \frac{C}{\eps^2} \| (\nabla u^\eps)_{\mathcal{S}} \|_{[L^2(\Omega_{\eps})]^9}\label{resc deru3},\quad \hbox{for } \alpha=1,2.
\end{align}
Therefore, from \eqref{resc disp}, \eqref{resc deru} and  \eqref{resc deru3} we get
$$\Pi_\eps (u^\eps_\alpha)  \rightharpoonup u_\alpha \textrm{ weakly in } H^1_\rho(\Omega),\quad \hbox{for } \alpha=1,2.$$

In the same way, from  \eqref{Korn inequality}$_1$, \eqref{Korn inequality}$_3$ and \eqref{der resc} we obtain
\begin{align}
\big\| \Pi_\eps u^\eps_3 \big\|_{L^2(\Omega)} & \leq C\frac{L}{\eps} \| (\nabla u)_{\mathcal{S}} \|_{[L^2(\Omega_{\eps})]^9},\label{resc u3}\\
  \Big\|\frac{\partial(\Pi_\eps u^\eps_3)}{\partial X_\beta} \Big\|_{L^2(\Omega)} &\leq \frac{C}{\eps} \| (\nabla u)_{\mathcal{S}} \|_{[L^2(\Omega_{\eps})]^9},  \textrm{ for } \beta=1,2,\label{resc du3}\\
   \Big\|\rho\frac{\partial(\Pi_\eps u^\eps_3)}{\partial x_3} \Big\|_{L^2(\Omega)} &  \leq \frac{C}{\eps} \| (\nabla u)_{\mathcal{S}} \|_{[L^2(\Omega_{\eps})]^9}\label{resc d3u3}.
\end{align}
Hence, we get 
$$\frac{1}{\eps}\Pi_\eps (u^\eps_3) \rightharpoonup u_3 \textrm{ weakly in } H^1_\rho(\Omega).$$

From the definition of the elementary displacement we have
$$u^\eps(x)-\mathcal{U}^\eps(x_3) =  \mathcal{R}^\eps(x_3) \land (x_1{\Ge}_1+x_2{\Ge}_2) + \bar{u}^\eps(x).$$
Hence, in view of \eqref{estimates1}$_1$, \eqref{resc warp}, the property \eqref{norms} of the rescaling operator and the assumption \eqref{assump} we obtain the following estimate:
\begin{equation}\label{estimate dif}
\frac{1}{\eps}\big\| \Pi_\eps (u^\eps - \mathcal{U}^\eps) \big\|_{L^2(\Omega; \R^3)} \leq \frac{1}{\eps^2}\Big\| \frac{\mathcal{R}^\eps(x_3) \land (x_1{\Ge}_1+x_2{\Ge}_2)}{\rho_\eps} \Big\|_{L^2(\Omega_\eps, \R^3)} + \frac{1}{\eps}\big\| \Pi_\eps \bar{u} \big\|_{L^2(\Omega; \R^3)} \leq C.
\end{equation}
Now using the rule of the derivation \eqref{der resc} and \eqref{norms} we have for $\alpha=1,2$
$$\frac{1}{\eps}\Big\| \frac{\partial\Pi_\eps (u^\eps - \mathcal{U}^\eps)}{\partial X_\alpha} \Big\|_{L^2(\Omega; \R^3)}  = \frac{1}{\eps}\Big\| \frac{\partial(u^\eps - \mathcal{U}^\eps)}{\partial x_\alpha} \Big\|_{L^2(\Omega_\eps; \R^3)}
\leq  \frac{1}{\eps}\Big\| \mathcal{R}^\eps(x_3) \land {\Ge}_\alpha \Big\|_{L^2(\Omega_\eps; \R^3)} + \frac{1}{\eps}\Big\| \frac{\partial \bar{u}}{\partial x_\alpha} \Big\|_{L^2(\Omega_{\eps}; \R^3)}\leq C,$$

\begin{align*}
\frac{1}{\eps}\Big\| \rho \frac{\partial\Pi_\eps (u^\eps - \mathcal{U}^\eps)}{\partial x_3} \Big\|_{L^2(\Omega; \R^3)} & \leq \frac{C}{\eps} \Big(\sum_{\alpha=1}^2\Big\| \frac{\partial(u^\eps - \mathcal{U}^\eps)}{\partial x_\alpha} \Big\|_{L^2(\Omega_\eps; \R^3)} \Big) + \frac{1}{\eps^2}\Big\| \frac{\partial(u^\eps - \mathcal{U}^\eps)}{\partial x_3} \Big\|_{L^2(\Omega_\eps; \R^3)}\\
& \leq C + \frac{1}{\eps^2}\Big\| \frac{d\mathcal{R}^\eps}{d x_3} \land (x_1\Ge_1+x_2\Ge_2) \Big\|_{L^2(\Omega_{\eps}; \R^3)} + \frac{1}{\eps^2} \Big\| \frac{\partial \bar{u}^\eps}{\partial x_3} \Big\|_{L^2(\Omega_{\eps}; \R^3)}\leq C.
\end{align*}
Consequently, from the two above estimates  and \eqref{estimate dif} we get the last weak convergence
$$ \frac{1}{\eps}\Pi_\eps (u^\eps-\mathcal{U}^\eps)  \rightharpoonup \mathcal{K} \textrm{ weakly in } H^1_\rho(\Omega;\R^3).$$

\medskip
\noindent{\it Step 2.} Relations between the limit fields.
\smallskip

Now we are going to establish the relations between the weak limits. First consider \eqref{estimates}$_2$ which implies
$$  \Big(\frac{d \mathcal{U}^\eps}{d x_3} - \mathcal{R}^\eps \land \Ge_{3}\Big) \to 0 \textrm{ strongly in }  L^2_\rho((0, L); \R^3),$$
as $\eps$ tends to $0$. Then \eqref{weak1} and \eqref{weak3} give
\begin{equation}\label{relation1}
\frac{d \mathcal{U}_1}{d x_3} = \mathcal{R}_2, \quad \frac{d \mathcal{U}_2}{d x_3} = - \mathcal{R}_1.
\end{equation}
It follows that $\mathcal{U}_\alpha \in H^2_{\rho^2}(0, L),$ for $\alpha=1,2.$
\medskip

\noindent Now, from \eqref{comp disp} we can write
\begin{equation}\label{limit u1} 
(\Pi_\eps u^\eps_1) (X_1, X_2, x_3) =   \mathcal{U}^\eps_{1} (x_3) - \eps \rho_\eps X_2 \mathcal{R}^\eps_3(x_3) + (\Pi_\eps \bar{u}^\eps_1)(X_1, X_2, x_3), \textrm{ for  a.e. } (X_1, X_2, x_3) \in \Omega.
\end{equation}
In view of \eqref{weak1}, \eqref{weak3}, \eqref{weak4} and \eqref{weak5} by passing to the limit in \eqref{limit u1} we obtain 
$$u_1(X_1, X_2, x_3) = \mathcal{U}_1(x_3), \textrm{ for a.e. } (X_1, X_2, x_3) \in \Omega.$$
Repeating the above arguments for $(\Pi_\eps u^\eps_2)$ we conclude that
$$u_2(X_1, X_2, x_3) = \mathcal{U}_2(x_3), \textrm{ for a.e. } (X_1, X_2, x_3) \in \Omega.$$
Notice that $u_\alpha$ does not depend on the variables $(X_1, X_2)$, for $\alpha=1,2$.
\smallskip

From \eqref{comp disp} we have for a.e. $(X_1, X_2, x_3) \in \Omega$
$$\frac{1}{\eps}(\Pi_\eps u^\eps_3) (X_1, X_2, x_3)=  \frac{1}{\eps} \mathcal{U}^\eps_{3}(x_3) - \rho_\eps X_1\mathcal{R}^\eps_2(x_3) + \rho_\eps X_2\mathcal{R}^\eps_1(x_3) + \frac{1}{\eps}(\Pi_\eps \bar{u}^\eps_3)(X_1, X_2, x_3).$$
Now, using \eqref{weak2}, \eqref{weak3}, \eqref{weak4} and \eqref{weak6} we pass to the limit in the equality above and we get
\begin{equation}\label{relation2}
u_3(X_1, X_2, x_3) = \mathcal{U}_{3}(x_3) - \rho X_1\mathcal{R}_2(x_3) + \rho X_2\mathcal{R}_1(x_3) \textrm{ for  a.e. } (X_1, X_2, x_3) \in \Omega.
\end{equation}
Observe that, due to \eqref{relation1}, \eqref{relation2} can be written as 
\begin{equation}\label{relation3}
u_3(X_1, X_2, x_3) = \mathcal{U}_{3}(x_3) - \rho X_1\frac{d \mathcal{U}_1}{d x_3}(x_3) - \rho X_2\frac{d \mathcal{U}_2}{d x_3}(x_3), \textrm{ for a.e. } (X_1, X_2, x_3) \in \Omega.
\end{equation}

Now we turn to the identification of $\mathcal{K}_i.$ In view of \eqref{comp disp} we have
\begin{equation}\label{limit K1}
\frac{1}{\eps}\Pi_\eps (u^\eps_1-\mathcal{U}^\eps_1)= - \rho_\eps X_2 \mathcal{R}^\eps_3(x_3) + \frac{1}{\eps}(\Pi_\eps \bar{u}^\eps_1)(X_1, X_2, x_3), \textrm{ for  a.e. }\; (X_1, X_2, x_3) \in \Omega.
\end{equation}
From \eqref{weak3}, \eqref{weak4}, \eqref{weak7} by passing to the limit in \eqref{limit K1} we obtain 
$$\mathcal{K}_1(X_1, X_2, x_3) = - \rho(x_3)X_2\mathcal{R}_3(x_3) \textrm{ for a.e. } (X_1, X_2, x_3) \in \Omega.$$
Proceeding as above for $\ds \frac{1}{\eps}\Pi_\eps (u^\eps_2-\mathcal{U}^\eps_2)$ we get
$$\mathcal{K}_2(X_1, X_2, x_3) =  \rho(x_3)X_1\mathcal{R}_3(x_3) \textrm{ for a.e. } (X_1, X_2, x_3) \in \Omega.$$
Finally we obtain the expression for $\mathcal{K}_3$. From  \eqref{comp disp} we have
$$\frac{1}{\eps}\Pi_\eps (u^\eps_3-\mathcal{U}^\eps_3)= -\rho_\eps X_1\mathcal{R}^\eps_2(x_3) + \rho_\eps X_2\mathcal{R}^\eps_1(x_3) + \frac{1}{\eps}(\Pi_\eps \bar{u}^\eps_3)(X_1, X_2, x_3).$$
Convergences \eqref{weak3}, \eqref{weak4}, \eqref{weak7}  allow to pass to the limit and we get
$$\mathcal{K}_3(X_1, X_2, x_3) = -\rho X_1\mathcal{R}_2(x_3) + \rho X_2\mathcal{R}_1(x_3)\textrm{ for a. e.  } (X_1, X_2, x_3) \in \Omega.$$
Equivalently, from \eqref{relation1} we have
$$\mathcal{K}_3(X_1, X_2, x_3) =  - \rho X_1\frac{d \mathcal{U}_1}{d x_3}(x_3)  - \rho X_2\frac{d \mathcal{U}_2}{d x_3}(x_3)\textrm{ for a.e. } (X_1, X_2, x_3) \in \Omega.$$
\vskip-9mm
\end{proof}
 
 \begin{remark} It is worth to note that the limit displacement fields is a kind of Bernoulli-Navier displacement.
 \smallskip
 
\noindent  Also observe that the limit warping $\bar{u}$ verifies the following conditions:
 \begin{equation}\label{warping cond1}
 \int_{\omega} \bar{u}_i \, dX_1 dX_2 = \int_{\omega} (X_1\bar{u}_2-X_2\bar{u}_1) \, dX_1 dX_2=\int_{\omega} X_\alpha\bar{u}_3 \, dX_1 dX_2=0, \quad i\in\{1, 2, 3\},\;\; \alpha\in\{1,2\}. 
 \end{equation}
  \end{remark}

To conclude this section, we give the asymptotic behavior of the gradient and  the symmetric gradient.
We define the field $\tilde{u}_3\in L^2((0,L); H^1(\omega))$ by setting
$$\tilde{u}_3(X_1, X_2, x_3) = \bar{u}_3(X_1, X_2, x_3) +\rho(x_3) \mathcal{Z}_1(x_3) X_1 +\rho(x_3) \mathcal{Z}_2(x_3) X_2\quad \hbox{for a.e. } (X_1,X_2,x_3)\in\Omega.$$
\begin{lemma}\label{lem43} In view of \eqref{weak1}-\eqref{weak7}  we obtain
 \begin{equation} \label{weak grad}
\Pi_\eps (\nabla u^\eps)  \rightharpoonup Z \textrm{ weakly in } [L^2_\rho(\Omega)]^9, \qquad \frac{1}{\eps}\Pi_\eps\big( (\nabla u^\eps)_{\mathcal{S}}\big)  \rightharpoonup T \textrm{ weakly in } [L^2_\rho(\Omega)]^9,
\end{equation}
where
\begin{equation*}
Z = \left(\begin{array}{ccc} 0 & \ds - \mathcal{R}_3  & \ds \mathcal{R}_2\\
[6mm] \ds  \mathcal{R}_3  & 0 & \ds -\mathcal{R}_1 \\
[6mm] \ds - \mathcal{R}_2  & \ds \mathcal{R}_1&0\end{array}\right),
\quad T= \left(\begin{array}{ccc} \ds \frac{\partial\bar{u}_1}{\partial X_1} & \ds  \frac{1}{2} \Big(\frac{\partial\bar{u}_1}{\partial X_2} +\frac{\partial\bar{u}_2}{\partial X_1}\Big)  & \ds \frac{1}{2}\Big( - \rho^2X_2\frac{d \mathcal{R}_3}{dx_3} + \frac{\partial\tilde{u}_3}{\partial X_1}\Big) \\
[6mm]  * & \ds \frac{\partial\bar{u}_2}{dX_2} & \ds \frac{1}{2}\Big(\rho^2X_1\frac{d \mathcal{R}_3}{dx_3} + \frac{\partial\tilde{u}_3}{\partial X_2}\Big) \\[6mm] * & *  & \ds \rho \frac{d\mathcal{U}_3}{dx_3} - \rho^2X_1\frac{d \mathcal{R}_2}{dx_3} +\rho^2X_2\frac{d \mathcal{R}_1}{dx_3} \end{array}\right).
\end{equation*}
\end{lemma}
\begin{proof}{\it Step 1.} Determination of the matrix $Z$.
\smallskip

In view of \eqref{comp grad} to obtain the $Z_{ij}$'s we only need to take into account the following convergences:
\begin{itemize}
\item From \eqref{resc warp}, \eqref{resc derwarp}, \eqref{resc derwarp3}, \eqref{rho} and \eqref{assump} we get
\begin{equation}\label{Eq 427-0}
{1\over \eps}\Pi_\eps \bar{u}^\eps_j \rightharpoonup 0 \textrm{ weakly in }  H^1_\rho(\Omega), \textrm{ for } j=1,2,3.
\end{equation} Hence
\begin{equation}\label{Eq 427}
{1\over \eps}\Pi_\eps \Big(\frac{\partial \bar{u}^\eps_j}{\partial x_3}\Big)\rightharpoonup 0 \textrm{ weakly in }  L^2_\rho(\Omega), \textrm{ for } j=1,2,3.
\end{equation}
\item Since $\mathcal{U}^\eps$ and $\mathcal{R}^\eps$ are independent of $x_1$ and $x_2$ we have
$$\Pi_\eps (\mathcal{R}^\eps) = \mathcal{R}^\eps, \quad \Pi_\eps (x_\alpha \mathcal{R}^\eps)= \eps \rho_\eps X_\alpha \mathcal{R}^\eps, \textrm{ for } \alpha=1,2, \quad  \Pi_\eps \Big( \frac{d\mathcal{U}^\eps}{dx_3}\Big) =  \frac{d\mathcal{U}^\eps}{dx_3}.$$
Then in view of \eqref{weak1}, \eqref{weak2}, \eqref{weak3} and \eqref{relation} we obtain
\begin{align*}
\Pi_\eps (\mathcal{R}^\eps) &\rightharpoonup \mathcal{R} \textrm{ weakly in } [L^2_\rho(\Omega)]^3,\\
\Pi_\eps (x_\alpha\mathcal{R}^\eps) &\to 0 \textrm{ strongly in } [L^2_\rho(\Omega)]^3, \textrm{ for } \alpha=1,2,\\
\Pi_\eps \Big( \frac{d\mathcal{U}^\eps_\alpha}{dx_3}\Big) &\rightharpoonup \frac{d \mathcal{U}_\alpha}{d x_3}= (-1)^{3-\alpha}\mathcal{R}_{3-\alpha} \textrm{ weakly in } L^2_\rho(\Omega), \textrm{ for } \alpha =1,2,\\
\Pi_\eps \Big( \frac{d\mathcal{U}^\eps_3}{dx_3}\Big) &\to 0 \textrm{ strongly in } L^2_\rho(\Omega).
\end{align*}
\end{itemize}

\noindent{\it Step 2.} Determination of the matrix $T$.
\smallskip

 From \eqref{der resc} we have
$$\frac{1}{\eps}\Pi_\eps \Big(\frac{\partial \bar{u}^\eps}{\partial x_\alpha} \Big)= \frac{1}{\eps^2 \rho_\eps} \frac{\partial(\Pi_\eps \bar{u}^\eps)}{\partial X_\alpha} \textrm{ for } \alpha=1,2.$$
Then in view of \eqref{sym grad} and using convergence \eqref{weak4} we obtain
$$T_{\alpha\beta} = \frac{1}{2} \Big(\frac{\partial\bar{u}_\alpha}{\partial X_\beta} +\frac{\partial\bar{u}_\beta}{\partial X_\alpha}\Big) \textrm{ for } \alpha, \beta=1,2.$$
Applying the rescaling operator to  \eqref{sym grad} we get
$$\frac{\rho}{\eps}\Pi_\eps \big( (\nabla u^\eps)_{\mathcal{S}}\big)_{13} =  \frac{1}{2}\Big[\frac{\rho}{\eps}\Big(\frac{d\mathcal{U}^\eps_1}{dx_3} -\mathcal{R}^\eps_2 \Big) - \rho \rho_\eps X_2\frac{d \mathcal{R}^\eps_3}{dx_3}+ \frac{\rho}{\rho_\eps}\frac{1}{\eps^2}\frac{\partial(\Pi_\eps \bar{u}^\eps_3)}{\partial X_1} +\frac{\rho}{\eps} \Pi_\eps \Big(\frac{\partial \bar{u}^\eps_1}{\partial x_3}\Big) \Big].$$
Convergences  \eqref{weak3}, \eqref{weak30},  \eqref{weak4} and \eqref{Eq 427} allow us to pass to the limit and we obtain
$$T_{13}= \frac{1}{2}\Big(\rho \mathcal{Z}_1 - \rho^2X_2\frac{d \mathcal{R}_3}{dx_3} + \frac{\partial\bar{u}_3}{\partial X_1}\Big)= \frac{1}{2}\Big(- \rho^2X_2\frac{d \mathcal{R}_3}{dx_3} + \frac{\partial\tilde{u}_3}{\partial X_1}\Big).$$
Similar calculations which are not repeated here allows us to get
$$T_{23}= \frac{1}{2}\Big(\rho \mathcal{Z}_2 + \rho^2X_1\frac{d \mathcal{R}_3}{dx_3} + \frac{\partial\bar{u}_3}{\partial X_2}\Big)=\frac{1}{2}\Big(  \rho^2X_1\frac{d \mathcal{R}_3}{dx_3} + \frac{\partial\tilde{u}_3}{\partial X_2}\Big).$$
To identify $T_{33}$ observe that from \eqref{sym grad} we have
$$\frac{\rho}{\eps}\Pi_\eps \big( (\nabla u^\eps)_{\mathcal{S}}\big)_{33} =  \frac{\rho}{\eps}\frac{d\mathcal{U}^\eps_3}{dx_3} -\rho \rho_\eps X_1\frac{d \mathcal{R}^\eps_2}{dx_3}+ \rho \rho_\eps X_2\frac{d \mathcal{R}^\eps_1}{dx_3} + \frac{\rho}{\eps} \Pi_\eps \Big(\frac{\partial \bar{u}^\eps_3}{\partial x_3}\Big).$$
Convergences  \eqref{weak2}, \eqref{weak3}, \eqref{weak4} and \eqref{Eq 427} allow us to pass to the limit and we obtain
$$T_{33}=\rho \frac{d\mathcal{U}_3}{dx_3} - \rho^2X_1\frac{d \mathcal{R}_2}{dx_3}+\rho^2X_2\frac{d \mathcal{R}_1}{dx_3}.$$
According to \eqref{relation}, $T_{33}$ can be expressed as
\begin{equation}\label{T33}
T_{33} = \rho \frac{d\mathcal{U}_3}{dx_3} - \rho^2X_1\frac{d^2 \mathcal{U}_1}{dx_3^2}-\rho^2X_2\frac{d^2\mathcal{U}_2}{dx_3^2}.
\end{equation}
\vskip-4mm
\end{proof}


\section{Position of the elastic problem}
We consider the standard linear isotropic equations of elasticity in $\Omega_\eps$.
The displacement field in $\Omega_\eps$ is denoted by
$$u^\eps: \Omega_\eps \to \R^3.$$
The linearized deformation field in  $\Omega_\eps$ is defined by
$$\gamma_{ij}(u^\eps) = \frac{1}{2}\Big({\partial u^\eps_j \over \partial x_i}+ {\partial u^\eps_i\over \partial x_j}\Big), \quad i,j=1,2,3.$$
The Cauchy stress tensor in $\Omega_\eps$ is linked to $\gamma(u^\eps)$ through the standard Hooke's law
$$\sigma^\eps_{ij} = \lambda \big(\sum_{k=1}^{3}\gamma_{kk}(u^\eps)\Big) \delta_{ij} + 2\mu\gamma_{ij}(u^\eps), \quad i,j=1,2,3,  $$
where $\lambda$ and $\mu$ denotes the Lame's coefficients of the elastic material and $\delta_{ij}=0$ if $i \neq j$ and $\delta_{ij}=1$ if $i=j$. The equation of equilibrium in $\Omega_\eps$ is
\begin{equation}\label{equilibrium}
-\sum_{j=1}^{3}{\partial \sigma^\eps_{ij}\over \partial x_j} = f_i^\eps \, \textrm{ in } \Omega_\eps, \quad i=1,2,3,
\end{equation}
where $f^\eps: \Omega_\eps \to \R^3$ denotes the applied force.
\medskip

We assume that the rod is clamped at the bottom, $\Gamma_{\eps,0}=\omega_{\eps,0} \times \{0\}$

$$u^\eps=0 \textrm{ on } \Gamma_{\eps,0},$$
and at the boundary $\partial \Omega_\eps \backslash \Gamma_{\eps,0}$ it is free
$$\sigma^\eps \nu_\eps = 0 \; \textrm{Êon } \partial \Omega_\eps \backslash \Gamma_{\eps,0},$$
where $\nu_\eps$ denotes the exterior unit normal to $\Omega_\eps$.
\medskip

Taking into account that the space of admissible displacements of the rod is
 $$H^1_{\Gamma_{\eps,0}}(\Omega_\eps; \R^3) = \{ u^\eps \in H^1(\Omega_\eps; \R^3)\; | \; u^\eps = 0 \textrm{ on } \Gamma_{\eps,0}\},$$
 the variational formulation of \eqref{equilibrium} is
 \begin{equation}\label{var form}
\left\{
\begin{aligned}
u^\eps \in H^1_{\Gamma_{\eps,0}}(\Omega_\eps; \R^3),&\\
\int_{\Omega_\eps} \sum_{i,j=1}^{3} \sigma^\eps_{ij} \gamma_{ij}(v) \, dx =& \int_{\Omega_\eps} \sum_{i=1}^{3} f^\eps_i v_i \,dx, \quad \forall v \in H^1_{\Gamma_{\eps,0}}(\Omega_\eps; \R^3).
\end{aligned}
\right. 
\end{equation}
For any $v \in H^1_{\Gamma_{\eps,0}}(\Omega_\eps; \R^3)$, the total elastic energy is denoted by
$$\mathcal{E}(v) = \int_{\Omega_\eps} \Big[\lambda \Big(\sum_{k=1}^{3}\gamma_{kk}(v)\Big)^2 + 2\mu\sum_{i,j=1}^{3}\big(\gamma_{ij}(v)\big)^2\Big] \, dx.$$

Observe that there exists a constant which depends only on $\lambda$ and $\mu$ such that for any $w \in H^1(\Omega_\eps; \R^3)$ we have
\begin{equation}\label{Eq 53}
C\|(\nabla w)_S\|^2_{[L^2(\Omega_\eps)]^9} \le \mathcal{E}(w).
\end{equation}

Taking $v=u^\eps$ in \eqref{var form} leads to the usual energy relation
\begin{equation}\label{energy}
\mathcal{E}(u^\eps) = \int_{\Omega_\eps} \sum_{i=1}^{3} f^\eps_i u^\eps_i \,dx.
\end{equation}

\subsection{Assumption on the forces}
In view of the energy relation \eqref{energy} and the estimates of the previous sections we assume throughout the paper
 \begin{equation}\label{applied forces}
\left\{
\begin{aligned}
F^\eps_1(x) &= \eps^2 f^\eps_1(x_3) -  x_2g^\eps_3(x_3), \textrm{ for } x \in \Omega_\eps\\
F^\eps_2(x) &= \eps^2 f^\eps_2(x_3) +  x_1 g^\eps_3(x_3),  \textrm{ for } x \in \Omega_\eps\\
F^\eps_3(x) &= \eps f^\eps_3(x_3) + x_1 g_1^\eps(x_3)+x_2g^\eps_2(x_3) \textrm{ for } x \in \Omega_\eps,
\end{aligned}
\right. 
\end{equation}
where $f^\eps, g^\eps \in L^2((0,L);\R^3)$ and they satisfy
\begin{equation}\label{F}
\big\| \rho_\eps^2 f^\eps   \big\|_{L^2((0, L);\R^3)} + \big\| \rho_\eps^3 g^\eps \big\|_{L^2((0, L);\R^3)} \leq C
\end{equation} the constant does not depend on $\eps$. Moreover we assume  the following weak convergences:
 \begin{equation}\label{weak conv forces}
\left\{
\begin{aligned}
f^\eps  &\longrightarrow f \enskip \textrm{ strongly in } L^2_{\rho^2}((0,L);\R^3),\\
g^\eps &\longrightarrow g \enskip\textrm{ strongly in } L^2_{\rho^3}((0,L);\R^3).
\end{aligned}
\right. 
\end{equation}

As a consequence, from \eqref{energy} and the relations \eqref{warping cond} we get an estimate of the total elastic energy
\begin{align*}
\mathcal{E}(u^\eps) = \int_0^L \Big[\eps^2 f^\eps_1(x_3) |\omega| \rho_\eps(x_3)^2\eps^2 \mathcal{U}^\eps_1(x_3) + \eps^2 f^\eps_2(x_3) |\omega| \rho_\eps(x_3)^2\eps^2\mathcal{U}^\eps_2(x_3) +  \eps f^\eps_3(x_3) |\omega| \rho_\eps(x_3)^2\eps^2\mathcal{U}^\eps_3(x_3)\Big]\, dx_3\\
+ \int_0^L \Big[ g^\eps_3(x_3) (I_1 + I_2) \rho_\eps(x_3)^4\eps^4 \mathcal{R}^\eps_3(x_3) + g^\eps_2(x_3) ( I_2) \rho_\eps(x_3)^4\eps^4 \mathcal{R}^\eps_1(x_3) - g^\eps_1(x_3) ( I_1) \rho_\eps(x_3)^4\eps^4 \mathcal{R}^\eps_2(x_3)\Big]\, dx_3
\end{align*}
Due to \eqref{estimates1}$_1$, \eqref{estimates1}$_4$, \eqref{estimates1}$_5$ and \eqref{Eq 53}-\eqref{F}  we have
$$
\begin{aligned}
\mathcal{E}(u^\eps) &\leq C \eps^2 \big( \big\| \rho_\eps^2 f^\eps   \big\|_{L^2((0, L);\R^3)} + \big\| \rho_\eps^3 g^\eps \big\|_{L^2((0, L);\R^3)}\big) \|(\nabla u^\eps)_S\|_{[L^2(\Omega_\eps)]^9} \leq C \eps^2  \mathcal{E}(u^\eps)^{1/2}. 
\end{aligned}
$$
That leads to
$$\mathcal{E}(u^\eps)^{1/2} \leq C \eps^2.$$
Hence
$$\|(\nabla u^\eps)_S\|_{[L^2(\Omega_\eps)]^9}\le C\eps^2.$$

\begin{remark}
Observe that the assumptions on the applied forces were assumed in order to obtain the appropriate estimate on the energy naturally.
\end{remark}

\section{The limit problems}

In this section we obtain the equations satisfied by the limit fields $\mathcal{U}$, $\mathcal{R}$ and $\bar{u}$. To do this, we assume that the forces are given by \eqref{applied forces} and satisfy \eqref{F}-\eqref{weak conv forces}.
First, we apply the rescaling operator $\Pi_\eps$ to the original variational formulation of the problem \eqref{var form} 
\begin{align}\label{rescvar}
\int_{\Omega}  \rho_\eps^2\sum_{i,j=1}^{3} \Pi_\eps(\sigma^\eps_{ij}) \Pi_\eps(\gamma_{ij}(v)) \, dX_1dX_2dx_3 =& \int_{\Omega} \rho_\eps^2 \sum_{i=1}^{3} \Pi_\eps(F^\eps_i) \Pi_\eps(v_i) \,dX_1dX_2dx_3, \quad \forall v \in H^1_{\Gamma_{\eps,0}}(\Omega_\eps; \R^3).
\end{align}
We will pass to the limit in \eqref{rescvar} as $\eps$ tends to zero. In order to accomplish this we need specific choices of the test function $v$.
We begin studying the behavior of the limit of the residual displacement $\bar{u}^\eps$.
\subsection{Equations for $\bar{u}$}
Let $\phi$ be in $H^1(\omega, \R^3)$ and $\varphi$ be in $C^{\infty}[0, L]$ such that $\varphi(0)=0$, we define the test function $v^\eps \in H^1_{\Gamma_{\eps,0}}(\Omega_\eps; \R^3)$
by
\begin{equation}\label{test function}
v^\eps(x_1, x_2, x_3) = \eps \varphi(x_3)\phi\Big(\frac{x_1}{\eps\rho_\eps},\frac{x_2}{\eps\rho_\eps}\Big), \quad (x_1,x_2,x_3) \in \Omega_\eps.
\end{equation}
Then we have
\begin{align*}
\gamma_{11}(v^\eps) &= \frac{1}{\rho_\eps}\varphi(x_3) \frac{\partial \phi_1}{\partial X_1}\Big(\frac{x_1}{\eps\rho_\eps},\frac{x_2}{\eps\rho_\eps}\Big),\\
\gamma_{12}(v^\eps) &= \frac{1}{2\rho_\eps}\varphi(x_3)\Big[ \frac{\partial \phi_1}{\partial X_2} +  \frac{\partial \phi_2}{\partial X_1}\Big]\Big(\frac{x_1}{\eps\rho_\eps},\frac{x_2}{\eps\rho_\eps}\Big),\\
\gamma_{13}(v^\eps) &= \frac{1}{2}\Big[-\varphi(x_3)\frac{\partial \phi_1}{\partial X_1} \frac{x_1\rho_\eps'}{\rho_\eps^2} - \varphi(x_3) \frac{\partial \phi_1}{\partial X_2} \frac{x_2\rho_\eps'}{\rho_\eps^2}+ \eps \varphi'(x_3)\phi_1 + \frac{1}{\rho_\eps} \varphi(x_3) \frac{\partial \phi_3}{\partial X_1}\Big]\Big(\frac{x_1}{\eps\rho_\eps},\frac{x_2}{\eps\rho_\eps}\Big)\\
&= - \frac{\varphi(x_3)\rho_\eps'}{2\rho_\eps^2}\Big[\sum_{\alpha=1}^2\frac{\partial \phi_1}{\partial X_\alpha}\Big(\frac{x_1}{\eps\rho_\eps},\frac{x_2}{\eps\rho_\eps}\Big)  x_\alpha\Big]+ \frac{\eps}{2}\varphi'(x_3)\phi_1\Big(\frac{x_1}{\eps\rho_\eps},\frac{x_2}{\eps\rho_\eps}\Big) + \frac{1}{2\rho_\eps} \varphi(x_3) \frac{\partial \phi_3}{\partial X_1}\Big(\frac{x_1}{\eps\rho_\eps},\frac{x_2}{\eps\rho_\eps}\Big),\\
\gamma_{22}(v^\eps) &= \frac{1}{\rho_\eps}\varphi(x_3) \frac{\partial \phi_2}{\partial X_2}\Big(\frac{x_1}{\eps\rho_\eps},\frac{x_2}{\eps\rho_\eps}\Big),\\
\gamma_{23}(v^\eps) &= - \frac{\varphi(x_3)\rho_\eps'}{2\rho_\eps^2}\Big[\sum_{\alpha=1}^2\frac{\partial \phi_2}{\partial X_\alpha}\Big(\frac{x_1}{\eps\rho_\eps},\frac{x_2}{\eps\rho_\eps}\Big)  x_\alpha\Big]+ \frac{\eps}{2}\varphi'(x_3)\phi_2\Big(\frac{x_1}{\eps\rho_\eps},\frac{x_2}{\eps\rho_\eps}\Big) + \frac{1}{2\rho_\eps} \varphi(x_3) \frac{\partial \phi_3}{\partial X_2}\Big(\frac{x_1}{\eps\rho_\eps},\frac{x_2}{\eps\rho_\eps}\Big),\\
\gamma_{33}(v^\eps) &=- \frac{\varphi(x_3)\rho_\eps'}{\rho_\eps^2}\Big[\frac{\partial \phi_3}{\partial X_1}\Big(\frac{x_1}{\eps\rho_\eps},\frac{x_2}{\eps\rho_\eps}\Big)  x_1 + \frac{\partial \phi_3}{\partial X_2}\Big(\frac{x_1}{\eps\rho_\eps},\frac{x_2}{\eps\rho_\eps}\Big)  x_2\Big] + \eps \varphi'(x_3)\phi_3\Big(\frac{x_1}{\eps\rho_\eps},\frac{x_2}{\eps\rho_\eps}\Big).
\end{align*}
Hence, using the properties of the rescaling operator we get the following strong convergences in $L^2(\Omega)$:
\begin{align}\label{test conv}
\rho_\eps \Pi_\eps(\gamma_{13}(v^\eps)) &\to \frac{1}{2}  \varphi(x_3) \frac{\partial \phi_3}{\partial X_1},\nonumber\\
\rho_\eps \Pi_\eps(\gamma_{23}(v^\eps)) &\to \frac{1}{2}  \varphi(x_3) \frac{\partial \phi_3}{\partial X_2},\nonumber\\
\rho_\eps \Pi_\eps(\gamma_{33}(v^\eps)) &\to 0.
\end{align}

\noindent Moreover, $\rho_\eps \Pi_\eps(\gamma_{11}(v^\eps)), \rho_\eps \Pi_\eps(\gamma_{12}(v^\eps))$ and $\rho_\eps \Pi_\eps(\gamma_{22}(v^\eps))$ are independent of $\eps$, since
\begin{align}\label{test conv1}
\rho_\eps \Pi_\eps(\gamma_{11}(v^\eps))(X_1, X_2, x_3) &=  \varphi(x_3) \frac{\partial \phi_1}{\partial X_1}(X_1, X_2),\nonumber \\
  \rho_\eps \Pi_\eps(\gamma_{12}(v^\eps))(X_1, X_2, x_3) &= 
 \frac{1}{2}\varphi(x_3)\Big[ \frac{\partial \phi_1}{\partial X_2} +  \frac{\partial \phi_2}{\partial X_1}\Big] (X_1, X_2),\nonumber \\
  \rho_\eps \Pi_\eps(\gamma_{22}(v^\eps)) (X_1, X_2, x_3) & = \varphi(x_3) \frac{\partial \phi_2}{\partial X_2}\Big(X_1,X_2\Big).
 \end{align}
 
\noindent  Now, we take $v^\eps$ as test function in \eqref{rescvar}, we have
$${1\over \eps}\Pi_\eps(v^\eps)(X_1, X_2, x_3) =  \varphi(x_3)\phi(X_1,X_2), \qquad\hbox{for}\;\; (X_1,X_2,x_3)\in \Omega.
$$
Then we pass to the limit. As far as the right hand side of \eqref{rescvar} is concerned, taking into account
 the assumptions \eqref{applied forces}-\eqref{F} and \eqref{weak conv forces} we have
  \begin{align*}
  \Pi_\eps(F_1^\eps) &= \eps^2 f^\eps_1(x_3) -  \rho_\eps \eps X_2g^\eps_3(x_3) \to 0 \quad \hbox{ strongly in } L^2_{\rho^2}(\Omega),\\
  \Pi_\eps(F_2^\eps) &= \eps^2 f^\eps_2(x_3) +  \rho_\eps \eps X_1g^\eps_3(x_3) \to 0 \quad \hbox{ strongly in } L^2_{\rho^2}(\Omega),\\
  \Pi_\eps(F_3^\eps) &= \eps f^\eps_3(x_3) +\rho_\eps \eps X_1 g_1(x_3)+ \rho_\eps \eps X_2g^\eps_2(x_3) \to 0 \quad \hbox{ strongly in } L^2_{\rho^2}(\Omega).
   \end{align*}
   Hence, dividing by $\eps$ the right hand side of \eqref{rescvar} and passing to the limit gives
   \begin{equation}\label{right side}
   \int_{\Omega} \frac{1}{\eps} \rho_\eps^2 \sum_{i=1}^{3} \Pi_\eps(F^\eps_i) \Pi_\eps(v^\eps_i) \,dX_1dX_2dx_3 \to 0. 
   \end{equation}
   
   On the other hand, using the convergences \eqref{test conv}, \eqref{test conv1} and \eqref{weak grad} we obtain the convergence of the left hand side (divided by $\eps$) when  $\eps$ goes to $0$
   \begin{align}\label{left side}
  & \int_{\Omega}  \frac{\rho_\eps^2}{\eps}\sum_{i,j=1}^{3} \Pi_\eps(\sigma^\eps_{ij}) \Pi_\eps(\gamma_{ij}(v)) \, dX_1dX_2dx_3\nonumber\\
  &  \to   \int_{\Omega} \Big\{(\lambda +2 \mu) \varphi \Big( \frac{\partial \bar u_1}{\partial X_1} \frac{\partial \phi_1}{\partial X_1}+  \frac{\partial \bar u_2}{\partial X_2} \frac{\partial \phi_2}{\partial X_2}\Big) +
   \lambda \varphi \Big(\frac{\partial \bar u_2}{\partial X_2}\frac{\partial \phi_1}{\partial X_1} + \frac{\partial \bar u_1}{\partial X_1} \frac{\partial \phi_2}{\partial X_2} \Big)\Big\} \, dXdx_3\nonumber\\
  &  + \int_{\Omega} \Big\{ \mu \varphi \Big( \frac{\partial \bar u_1}{\partial X_2} + \frac{\partial \bar u_2}{\partial X_1}\Big) \Big( \frac{\partial \phi_1}{\partial X_2} + \frac{\partial \phi_2}{\partial X_1}\Big)
   + \lambda \varphi \Big(\rho \frac{d\mathcal{U}_3}{dx_3} - \rho^2X_1\frac{d^2 \mathcal{U}_1}{dx_3^2}-\rho^2X_2\frac{d^2\mathcal{U}_2}{dx_3^2} \Big)\Big(\frac{\partial \phi_1}{\partial X_1} +\frac{\partial \phi_2}{\partial X_2} \Big)\Big\} \, dXdx_3 \nonumber \\
  & + \int_{\Omega} \Big\{\mu \varphi \frac{\partial \phi_3}{\partial X_1} \Big( - \rho^2X_2\frac{d \mathcal{R}_3}{dx_3} + \frac{\partial\tilde{u}_3}{\partial X_1}\Big) + \mu \varphi \frac{\partial \phi_3}{\partial X_2} \Big( \rho^2X_1\frac{d \mathcal{R}_3}{dx_3} + \frac{\partial\tilde{u}_3}{\partial X_2}\Big) \Big\} \, dXdx_3
\end{align}
%
 where  $\phi$ be in $H^1(\omega, \R^3)$ and $\varphi$ be in $C^{\infty}[0, L]$ such that $\varphi(0)=0$. Due to the convergence \eqref{right side}, the above limit is equal to zero. Since $\varphi$ is arbitrary, we can localized with respect to $x_3$; that gives
 \begin{align}\label{limit}
  & \int_{\omega} \Big\{(\lambda +2 \mu) 
   \Big( \frac{\partial \bar u_1}{\partial X_1} \frac{\partial \phi_1}{\partial X_1}+  \frac{\partial \bar u_2}{\partial X_2} \frac{\partial \phi_2}{\partial X_2}\Big) +
   \lambda  \Big(\frac{\partial \bar u_2}{\partial X_2}\frac{\partial \phi_1}{\partial X_1} + \frac{\partial \bar u_1}{\partial X_1} \frac{\partial \phi_2}{\partial X_2} \Big)\Big\} \, dX_1 dX_2 \nonumber\\
  &  + \int_{\omega} \Big\{ \mu  \Big( \frac{\partial \bar u_1}{\partial X_2} + \frac{\partial \bar u_2}{\partial X_1}\Big) \Big( \frac{\partial \phi_1}{\partial X_2} + \frac{\partial \phi_2}{\partial X_1}\Big)
   + \lambda  \Big(\rho \frac{d\mathcal{U}_3}{dx_3} - \rho^2X_1\frac{d^2 \mathcal{U}_1}{dx_3^2}-\rho^2X_2\frac{d^2\mathcal{U}_2}{dx_3^2} \Big)\Big(\frac{\partial \phi_1}{\partial X_1} +\frac{\partial \phi_2}{\partial X_2} \Big)\Big\} \, dX_1 dX_2  \nonumber \\
  & + \int_{\omega} \Big\{\mu  \frac{\partial \phi_3}{\partial X_1} \Big( - \rho^2X_2\frac{d \mathcal{R}_3}{dx_3} + \frac{\partial\tilde{u}_3}{\partial X_1}\Big) + \mu  \frac{\partial \phi_3}{\partial X_2} \Big( \rho^2X_1\frac{d \mathcal{R}_3}{dx_3} + \frac{\partial\tilde{u}_3}{\partial X_2}\Big) \Big\} \, dX_1 dX_2=0
\end{align} 

\subsubsection{Determination of $\tilde u_3$}

\noindent  First, we choose $\phi_1 = \phi_2=0$. In view of \eqref{limit} we have 
 \begin{equation*}
\int_{\omega} \Big\{ \frac{\partial \phi_3}{\partial X_1} \Big( - \rho^2X_2\frac{d \mathcal{R}_3}{dx_3} + \frac{\partial\tilde{u}_3}{\partial X_1}\Big) + \frac{\partial \phi_3}{\partial X_2} \Big(\rho^2X_1\frac{d \mathcal{R}_3}{dx_3} + \frac{\partial\tilde{u}_3}{\partial X_2}\Big) \Big\} \, dX_1 dX_2 = 0, \quad \hbox {a.e. in } [0,L]. 
 \end{equation*}    
 Then the field $\tilde{u}_3\in L^2((0,L); H^1(\omega))$ satisfies 
\begin{equation}\label{tilde u3}
\int_{\omega} \nabla_{X}\tilde{u}_3 \nabla_X\phi_3 \, dX = - \rho^2\frac{d \mathcal{R}_3}{dx_3} \int_{\omega} \Big\{-X_2 \frac{\partial \phi_3}{\partial X_1} + X_1 \frac{\partial \phi_3}{\partial X_2}\Big\} \, dX.
\end{equation}
Now, we introduce the function  $\chi$ as the unique solution of the following torsion problem:
  \begin{equation} \label{torsion problem}
\left\{
\begin{gathered}
\chi \in H^1(\omega), \quad \int_{\omega} \chi \, dX =0,\\
 \int_{\omega} \nabla_{X}\chi \nabla_X\psi \, dX = -  \int_{\omega} \Big\{-X_2 \frac{\partial \psi}{\partial X_1} + X_1 \frac{\partial \psi}{\partial X_2}\Big\} \, dX, \quad 
 \forall \psi \in H^1(\omega),
\end{gathered}
\right. 
\end{equation}
Taking $\chi$ as test function in \eqref{torsion problem} gives
$$\|\nabla\chi\|^2_{[L^2(\omega)]^2}\le I_1+I_2.$$ By contradiction, we easily prove   $\|\nabla\chi\|^2_{[L^2(\omega)]^2} < I_1+I_2$. We set
\begin{equation}\label{K}
K=I_1 + I_2+ \int_{\omega} \Big\{ - X_2\frac{\partial \chi}{\partial X_1}+X_1\frac{\partial \chi}{\partial X_2}\Big\} \, dX_1 dX_2=I_1 + I_2-\|\nabla\chi\|^2_{[L^2(\omega)]^2}>0.
\end{equation} The above constant which depends on the geometry of the reference cross section $\omega$, is the St Venant torsional stiffness.

Since $\tilde{u}_3$ verifies \eqref{tilde u3} and also $\ds \int_{\omega}\tilde{u}_3(X_1,X_2,x_3) \, dX_1 dX_2=0$ for a.e.  $x_3$ in $(0, L)$,  we get 
$$\tilde{u}_3(X_1, X_2, x_3) = \chi (X_1, X_2)\rho^2(x_3)\frac{d \mathcal{R}_3}{dx_3}(x_3)\qquad \hbox{for a.e.} (X_1,X_2,x_3)\in \Omega$$
which in turn gives
\begin{equation}\label{T13T23}
T_{13}=\Big(-X_2+ \frac{\partial \chi}{\partial X_1}\Big) \frac{\rho^2}{2}\frac{d \mathcal{R}_3}{dx_3}, \qquad T_{23}=\Big(X_1+ \frac{\partial \chi}{\partial X_2}\Big) \frac{\rho^2}{2}\frac{d \mathcal{R}_3}{dx_3}.
\end{equation}
\subsubsection{Determination of $\bar u_\alpha$, $\alpha=1,2$}
\smallskip

Now taking $\phi_3 = 0$ in \eqref{limit}  yields
 \begin{align}\label{limit1}
   & \int_{\omega} \Big\{(\lambda +2 \mu) 
   \Big( \frac{\partial \bar u_1}{\partial X_1} \frac{\partial \phi_1}{\partial X_1}+  \frac{\partial \bar u_2}{\partial X_2} \frac{\partial \phi_2}{\partial X_2}\Big) +
   \lambda  \Big(\frac{\partial \bar u_2}{\partial X_2}\frac{\partial \phi_1}{\partial X_1} + \frac{\partial \bar u_1}{\partial X_1} \frac{\partial \phi_2}{\partial X_2} \Big)\Big\} \,dX  \nonumber\\
   & +  \int_{\omega} \Big\{\mu  \Big( \frac{\partial \bar u_1}{\partial X_2} + \frac{\partial \bar u_2}{\partial X_1}\Big) \Big( \frac{\partial \phi_1}{\partial X_2} + \frac{\partial \phi_2}{\partial X_1}\Big)\Big\} \,dX\nonumber\\
  &= - \int_{\omega} \Big\{\lambda  \Big(\rho \frac{d\mathcal{U}_3}{dx_3} - \rho^2X_1\frac{d^2 \mathcal{U}_1}{dx_3^2} - \rho^2X_2\frac{d^2 \mathcal{U}_2}{dx_3^2} \Big)\Big(\frac{\partial \phi_1}{\partial X_1} + \frac{\partial \phi_2}{\partial X_2} \Big) \Big\} \, dX\qquad \hbox{a.e. in }(0,L) 
 \end{align}
 for any  $\phi_\alpha \in H^1(\omega) (\alpha=1, 2)$. Then the variational problem \eqref{limit1} corresponds to a classical 2d elastic problem for 
 $(\bar{u}_1, \bar{u}_2).$ Taking into account the relations \eqref{warping cond1}, the above variational problem admits a unique solution. Then  we obtain
\begin{align}
&{\partial\bar{u}_1\over \partial X_1}(X_1,X_2,\cdot)= -\nu \Big\{\rho \frac{d\mathcal{U}_3}{dx_3} - \rho^2 X_1 \frac{d^2 \mathcal{U}_1}{dx_3^2} - \rho^2 X_2\frac{d^2 \mathcal{U}_2}{dx_3^2}\Big\},\\
&{\partial\bar{u}_2\over \partial X_2}(X_1,X_2,\cdot)= -\nu \Big\{\rho \frac{d\mathcal{U}_3}{dx_3} - \rho^2 X_1 \frac{d^2 \mathcal{U}_1}{dx_3^2} - \rho^2 X_2\frac{d^2 \mathcal{U}_2}{dx_3^2}\Big\},\\
&\Big({\partial\bar{u}_1\over \partial X_2}-{\partial\bar{u}_2\over \partial X_1}\Big)(X_1,X_2,\cdot)=0\qquad \hbox{a.e. in}\;\; \Omega
\end{align}
where $\ds \nu = \frac{\lambda}{2(\lambda + \mu)}$ is the Poisson coefficient of the material.

As a consequence we get 
\begin{equation}\label{T11T22}
T_{12} = 0, \qquad T_{11}= T_{22} = -\nu \Big\{\rho \frac{d\mathcal{U}_3}{dx_3} - \rho^2 X_1 \frac{d^2 \mathcal{U}_1}{dx_3^2} - \rho^2 X_2\frac{d^2 \mathcal{U}_2}{dx_3^2}\Big\}.
\end{equation}

\subsection{Equations for $\mathcal{U}_1, \mathcal{U}_2$ and $\mathcal{R}_3$}
Now we consider the functions $\varphi_1, \varphi_2$ and $\varphi_3$ in $C^{\infty}[0, L]$ such that $\varphi_1(0) = \varphi_1'(0) = \varphi_2(0) = \varphi_2'(0) = \varphi_3(0)=0 .$
We construct a test field $\phi^\eps \in  H^1_{\Gamma_{\eps,0}}(\Omega_\eps; \R^3)$ as follows:
\begin{align*}
\phi^\eps_1(x) &= \varphi_1(x_3) - x_2\varphi_3(x_3),\\
\phi^\eps_2(x) &= \varphi_2(x_3) + x_1\varphi_3(x_3),\\
\phi^\eps_3(x) &= -x_1\varphi_1'(x_3) - x_2\varphi_2'(x_3).
\end{align*}
Then we get
\begin{align*}
\gamma_{11}(\phi^\eps) &= 0, \, \gamma_{22}(\phi^\eps) = 0, \, \gamma_{12}(\phi^\eps) = 0,\\
 \gamma_{13}(\phi^\eps) &= -\frac{1}{2} x_2 \varphi_3'(x_3),\\
 \gamma_{23}(\phi^\eps) &= \frac{1}{2} x_1 \varphi_3'(x_3),\\
 \gamma_{33}(\phi^\eps) &=- x_1 \varphi_1''(x_3) -  x_2 \varphi_2''(x_3).\\
 \end{align*}
Applying the rescaling operator $\Pi_\eps$ to the previous expressions gives
\begin{align*}
\Pi_\eps\big(\gamma_{11}(\phi^\eps)\big) &= \Pi_\eps\big(\gamma_{22}(\phi^\eps)\big) = \Pi_\eps \big( \gamma_{12}(\phi^\eps)\big) = 0,\\
 \Pi_\eps\big(\gamma_{13}(\phi^\eps)\big)&= -\frac{1}{2} \eps \rho_\eps X_2 \varphi_3'(x_3),\\
  \Pi_\eps\big(\gamma_{23}(\phi^\eps)\big)&= \frac{1}{2}\eps \rho_\eps X_1 \varphi_3'(x_3),\\
  \Pi_\eps\big( \gamma_{33}(\phi^\eps)\big)&= - \eps \rho_\eps X_1 \varphi_1''(x_3) -\eps \rho_\eps X_2 \varphi_2''(x_3).\\
 \end{align*}
 In order to obtain the limit problem as $\eps$ tends to $0$, we consider $ v = \phi^\eps$ in \eqref{rescvar}, it leads to
 \begin{align}\label{var problem}
\int_{\Omega} \rho_\eps^2\sum_{i,j=1}^{3} \Pi_\eps(\sigma^\eps_{ij}) \Pi_\eps(\gamma_{ij}(\phi^\eps)) \, dX_1dX_2dx_3 =& \int_{\Omega} \rho_\eps^2 \sum_{i=1}^{3} \Pi_\eps(F^\eps_i) \Pi_\eps(\phi^\eps_i) \,dX_1dX_2dx_3. 
\end{align}
We divide the above equality by $\eps^2$. Then using the convergence \eqref{weak grad} and the definition of the test function we can pass to the limit
in the left-hand side to obtain
\begin{align*}
&\lim_{\eps \to 0} \int_{\Omega} \frac{\rho_\eps^2}{\eps^2}\sum_{i,j=1}^{3} \Pi_\eps(\sigma^\eps_{ij}) \Pi_\eps(\gamma_{ij}(\phi^\eps)) \, dX_1dX_2dx_3 \\
&=  \mu \int_{\Omega} \rho^2 \Big[-X_2 \varphi_3'(x_3) T_{13}+ X_1 \varphi_3'(x_3)T_{23}\Big]\, dX_1dX_2dx_3\\
&+\int_{\Omega} \rho^2 \Big[  \Big( - \sum_{\alpha=1}^2 X_\alpha \varphi_\alpha''(x_3) \Big) \Big( (\lambda + 2\mu)\Big(\rho \frac{d\mathcal{U}_3}{dx_3} - \rho^2X_1\frac{d^2 \mathcal{U}_1}{dx_3^2} - \rho^2X_2\frac{d^2 \mathcal{U}_2}{dx_3^2}\Big) + \lambda \Big(\frac{\partial\bar{u}_1}{\partial x_1} +\frac{\partial\bar{u}_2}{\partial x_2}\Big) \Big)\Big]\, dX_1dX_2dx_3.
\end{align*}
Moreover, taking into account \eqref{T13T23} and \eqref{T11T22}, the above limit is equal to 
\begin{align}\label{left}
& \int_{\Omega} \rho^4\frac{d \mathcal{R}_3}{dx_3}  \Big[- \frac{\mu}{2} X_2 \varphi_3'(x_3) \Big(-X_2+ \frac{\partial \chi}{\partial X_1}\Big)
+\frac{\mu}{2} X_1 \varphi_3'(x_3)\Big(X_1+ \frac{\partial \chi}{\partial X_2}\Big) \Big]\, dX_1dX_2dx_3 \nonumber \\
&+ \int_{\Omega} \rho^2 \Big[  E \Big(  -X_1 \varphi_1''(x_3) - X_2 \varphi_2''(x_3) \Big)\Big(\rho \frac{d\mathcal{U}_3}{dx_3} - \rho^2X_1\frac{d^2 \mathcal{U}_1}{dx_3^2} - \rho^2X_2\frac{d^2 \mathcal{U}_2}{dx_3^2}\Big)\Big]\, dX_1dX_2dx_3,
\end{align}
where $\ds E = \frac{ \mu (3\lambda + 2\mu)}{\lambda + \mu}$ is the Young's modulus of the elastic material.

On the other hand, in view of the assumptions \eqref{applied forces}, \eqref{weak conv forces} and the definition of the test field we obtain the following limit for the right-hand side:
\begin{align}\label{right}
&\lim_{\eps \to 0}\int_{\Omega} \frac{\rho_\eps^2}{\eps^2} \sum_{i=1}^{3} \Pi_\eps(F^\eps_i) \Pi_\eps(\phi^\eps_i) \,dX_1dX_2dx_3  \nonumber\\
&= \int_{\Omega} \rho^2\Big\{f_1 \varphi_1 + \rho^2 X_2^2 g_3\varphi_3 + f_2 \varphi_2 + \rho^2 X_1^2 g_3 \varphi_3 - \rho^2 X_1^2 g_1 \varphi_1' - \rho^2 X_2^2 g_2 \varphi_2'\Big\} \, dX_1dX_2dx_3.
\end{align}

\noindent Hence, by \eqref{left} and \eqref{right} the limit equation of \eqref{var problem} is given by
\begin{align}\label{limit2}
& \int_{\Omega} \rho^4\frac{d \mathcal{R}_3}{dx_3}  \Big[- \frac{\mu}{2} X_2 \varphi_3'(x_3) \Big(-X_2+ \frac{\partial \chi}{\partial X_1}\Big)
+\frac{\mu}{2} X_1 \varphi_3'(x_3)\Big(X_1+ \frac{\partial \chi}{\partial X_2}\Big) \Big]\, dX_1dX_2dx_3 \nonumber \\
&+\int_{\Omega} \rho^2 \Big[  E \Big( -X_1 \varphi_1''(x_3) - X_2 \varphi_2''(x_3) \Big)\Big(\rho \frac{d\mathcal{U}_3}{dx_3} - \rho^2X_1\frac{d^2 \mathcal{U}_1}{dx_3^2} - \rho^2X_2\frac{d^2 \mathcal{U}_2}{dx_3^2}\Big)\Big]\, dX_1dX_2dx_3 \nonumber \\
&= \int_{\Omega} \rho^2\Big\{f_1 \varphi_1 + \rho^2 X_2^2 g_3\varphi_3 + f_2 \varphi_2 + \rho^2 X_1^2 g_3 \varphi_3 - \rho^2 X_1^2 g_1 \varphi_1' - \rho^2 X_2^2 g_2 \varphi_2'\Big\} \, dX_1dX_2dx_3,
\end{align}
for any  $\varphi_3 \in C^{\infty}[0, L]$ such that $\varphi_3(0)=0 $ and for $\varphi_1, \varphi_2 \in C^{\infty}[0, L]$ such that  $\varphi_1(0) = \varphi_1'(0) = \varphi_2(0) = \varphi_2'(0) = 0.$ We simplify \eqref{limit2} 
\begin{align}\label{varform}
  &\frac{K\mu}{2}\int_{(0, L)}  \rho^4 \frac{d \mathcal{R}_3}{dx_3}\varphi'_3 dx_3 + EI_1\int_{(0, L)}  \rho^4 \frac{d^2 \mathcal{U}_1}{dx_3^2} \varphi''_1\,dx_3 +EI_2\int_{(0, L)}  \rho^4  \frac{d^2 \mathcal{U}_2}{dx_3^2} \varphi''_2\,dx_3 \nonumber \\  
 = &(I_1+I_2)\int_{(0, L)} \rho^4 g_3\varphi_3  dx_3+ |\omega|\int_{(0,L)} \rho^2\big\{f_1 \varphi_1 + f_2 \varphi_2 \big\}dx_3-\sum_{\alpha=1}^2I_\alpha\int_{(0,L)}  \rho^4 g_\alpha \varphi'_\alpha\,dx_3.
\end{align}
\smallskip

First we choose $\varphi_1= \varphi_2 = 0$ in \eqref{varform}. Taking into account the boundary condition $\mathcal{R}_3(0) = 0$, the function $\mathcal{R}_3$ is the unique solution  of 
  \begin{equation} \label{equation R3}
\left\{
\begin{gathered}
- \frac{K\mu}{2} \frac{d}{dx_3}\Big(\rho^4 \frac{d \mathcal{R}_3}{dx_3}\Big) = (I_1 + I_2) \rho^4 g_3\\
\mathcal{R}_3(0)=0,
\end{gathered}
\right. 
\end{equation}
where $K$ is given by \eqref{K}.

In  \eqref{varform} we take $\varphi_3=0$. Since $\varphi_1$ and $\varphi_2$ are arbitrary in $C^\infty[0, L]$ such that $\varphi_1(0) = \varphi_1'(0) = \varphi_2(0) = \varphi_2'(0) = 0,$ that gives the bending problems satisfied by  $\mathcal{U}_1$ and $\mathcal{U}_2$
\begin{equation} \label{bending}
\left\{
\begin{gathered}
 E I_\alpha \frac{d^2}{d x_3^2}\Big(\rho^4 \frac{d^2 \mathcal{U}_\alpha}{d x_3^2} \Big) = |\omega|\rho^2f_\alpha + I_\alpha \frac{d}{d x_3}\big(\rho^4 g_\alpha\big) ,\\
 \mathcal{U}_\alpha(0)=\frac{d \mathcal{U}_\alpha}{d x_3}(0)=0,
  \end{gathered}
\right. \qquad \hbox{for } \alpha=1,2.
\end{equation}

Recall that in order to obtain \eqref{equation R3}-\eqref{bending}, we have used the fact that $\rho(L)=0$.

\subsection{Equation for $\mathcal{U}_3$}
In this step we derive the equation satisfied by $\mathcal{U}_3$. In order to get this, in \eqref{rescvar} we consider as test field  $v(x_1, x_2, x_3) = (0,0, \varphi(x_3))$ in $H^1(\Omega_\eps; \R^3)$ such that $\varphi \in C^\infty[0, L]$ with
$\varphi(0)=0.$
Due to the assumptions \eqref{applied forces}, \eqref{weak conv forces}, the definition of the test field $v$ and taking into account \eqref{integral} the limit of \eqref{rescvar} devided by $\eps$ gives
\begin{equation*}
\int_{(0, L)} E \rho^2 \frac{d\mathcal{U}_3}{dx_3} \varphi_3'\,dx_3 = \int_{(0, L)} \rho^2f_3\varphi_3 \,dx_3. 
\end{equation*}
Hence, since $\varphi$ is any function in $C^\infty[0, L]$ such that $\varphi(0)=0$ and $\rho(L)=0$ we can conclude that $\mathcal{U}_3$ verifies the following compression-traction equation for elastic rods:
\begin{equation} \label{compression}
\left\{
\begin{gathered}
- E \frac{d}{d x_3}\Big(\rho^2 \frac{d \mathcal{U}_3}{d x_3} \Big) =\rho^2 f_3,\\
 \mathcal{U}_3(0)=0.
 \end{gathered}
\right. 
\end{equation}

\subsection{Convergence of the total elastic energy}
In the above subsections all the limit problems admit a unique solution.  As a consequence the whole sequences $\ds \Big\{{1\over \eps^2}\overline{u}^\eps\Big\}_\eps$,   $\{{\cal U}^\eps_\alpha\}_\eps$, $\ds\Big \{{1\over \eps}{\cal U}^\eps_3\Big\}_\eps$ and $\{{\cal R}^\eps_3\}_\eps$ converge weakly to their limit.

In this subsection we prove that the rescaled energy $\ds \frac{\mathcal{E}(u^\eps)}{\eps^4}$ converges to the elastic limit energy as $\eps$ tends to zero and that some weak convergences are in fact strong convergences.

\begin{lemma} Under the assumptions \eqref{applied forces}, \eqref{F} and \eqref{weak conv forces} on the applied forces, we obtain the following convergence for the total elastic energy
\begin{equation}\label{energy1}
\lim_{\eps \to 0}\frac{\mathcal{E}(u^\eps)}{\eps^4}= \int_{\Omega} \Big\{\lambda Tr(T) Tr(T) + \sum_{i,j=1}^3 2\mu T_{ij} T_{ij}\Big\} dX_1dX_2dx_3,
\end{equation}
where $T$ is the limit of the symmetric gradient defined in \eqref{weak grad}.
\end{lemma}
\begin{proof}
Taking $v=u^\eps$ in \eqref{var form}, dividing by $\eps^4$, then using   the properties of $\Pi_\eps$ and by standard weak lower-semi-continuity, we obtain
\begin{equation}\label{liminf}
\int_{\Omega} \Big\{\lambda Tr(T) Tr(T) + \sum_{i,j=1}^3 2\mu T_{ij} T_{ij}\Big\} dX_1dX_2dx_3\le \liminf_{\eps \to 0}\frac{\mathcal{E}(u^\eps)}{\eps^4}.
\end{equation} We have
$$\frac{\mathcal{E}(u^\eps)}{\eps^4}=\int_{\Omega} \frac{ \rho_\eps^2}{\eps^2}\sum_{i,j=1}^{3} \Pi_\eps(\sigma^\eps_{ij}) \Pi_\eps(\gamma_{ij}(u^\eps)) \, dX_1dX_2dx_3 =\int_{\Omega} \frac{\rho_\eps^2}{\eps^2} \sum_{i=1}^{3} \Pi_\eps(F^\eps_i) \Pi_\eps(u^\eps_i) \,dX_1dX_2dx_3.$$
The last term in the above equality is equal to
\begin{align*}
&\int_{\Omega} \frac{\rho_\eps^2}{\eps^2} \sum_{i=1}^{3} \Pi_\eps(F^\eps_i) \Pi_\eps(u^\eps_i) \,dX_1dX_2dx_3 = \sum_{\alpha=1}^{2} \int_{\Omega} \rho_\eps^2 f^\eps_\alpha(x_3)\Pi_\eps(u^\eps_\alpha)\,dX_1dX_2dx_3 - \int_{\Omega} \frac{\rho_\eps^3}{\eps} X_{2}g^\eps_3(x_3)\Pi_\eps(u^\eps_1)\\ &+ \int_{\Omega} \frac{\rho_\eps^3}{\eps} X_{1}g^\eps_3(x_3)\Pi_\eps(u^\eps_2) + \int_{\Omega} \frac{\rho_\eps^2}{\eps} f^\eps_3(x_3)\Pi_\eps(u^\eps_3)\,dX_1dX_2dx_3 + \sum_{\alpha=1}^{2} \int_{\Omega} \frac{\rho_\eps^3}{\eps} X_\alpha g_\alpha^\eps(x_3)\Pi_\eps(u^\eps_3)\,dX_1dX_2dx_3.
\end{align*}
Then \eqref{resc warp}, \eqref{weak3}, \eqref{weak5}, \eqref{weak6}, \eqref{u3} and \eqref{weak conv forces} lead to
\begin{align}\label{limit energy1}
\limsup_{\eps \to 0}\frac{\mathcal{E}(u^\eps)}{\eps^4}
= & \int_{\Omega}\Big[\sum_{i=1}^{3}\rho^2 f_i u_i + \sum_{\alpha=1}^{2}\rho^4 \big(X_\alpha^2g_3\mathcal{R}_3 -  X_\alpha^2g_\alpha \frac{d \mathcal{U}_\alpha}{d x_3}\big)\Big]\,dX_1dX_2dx_3\nonumber\\
= & |\omega|\int_{(0,L)} \rho^2 f\cdot {\cal U}\,dx_3+(I_1+I_2)\int_{(0, L)} \rho^4 g_3{\cal R}_3  dx_3\nonumber\\
-& I_1\int_{(0,L)}  \rho^4 g_1 {d{\cal U}_1\over dx_3}\,dx_3 -I_2\int_{(0,L)} \rho^4 g_2 {d{\cal U}_2\over dx_3}\,dx_3.
\end{align}

Besides, since $T$ is a symmetric matrix we know that it verifies the following algebraic identity
\begin{align*}
\lambda Tr(T) Tr(T) + \sum_{i,j=1}^3 2\mu T_{ij} T_{ij} &= E T_{33}^2 +  \frac{E}{(1+\nu)(1-2\nu)}(T_{11}+T_{22}+2\nu T_{33})^2\\
 &+\frac{E}{2(1+\nu)}[(T_{11}-T_{22})^2+
4(T_{12}^2+T_{13}^2+T_{23}^2)].
\end{align*}
Then, in view of \eqref{integral}, \eqref{T33}, \eqref{T13T23}, \eqref{T11T22} and \eqref{torsion problem} we have
\begin{align}\label{limit energy}
&\int_{\Omega}\lambda Tr(T) Tr(T) + \sum_{i,j=1}^3 2\mu T_{ij} T_{ij} \,dX_1dX_2dx_3 = E \int_{\Omega} \rho^2 \Big(\frac{d \mathcal{U}_3}{d x_3} \Big)^2 \,dX_1dX_2dx_3\nonumber\\
&+E \int_{0}^L\rho^4\Big( I_1 \Big(\frac{d\mathcal{R}_2}{dx_3}\Big)^2 + I_2 \Big(\frac{d\mathcal{R}_1}{dx_3}\Big)^2\Big)\,dX_1dX_2dx_3 +\frac{K\mu}{2} \int_{0}^L\rho^4\Big(\frac{d\mathcal{R}_3}{dx_3}\Big)^2\,dx_3\nonumber\\
&= E |\omega|\int_{\Omega} \rho^2 \Big(\frac{d \mathcal{U}_3}{d x_3} \Big)^2 \,dx_3+ \sum_{\alpha=1}^2 EI_\alpha\int_{(0, L)}  \rho^4 \Big(\frac{d^2 \mathcal{U}_\alpha}{dx_3^2}\Big)^2\,dx_3 +\frac{K\mu}{2}\int_{(0, L)}  \rho^4 \Big(\frac{d \mathcal{R}_3}{dx_3}\Big)^2 dx_3. 
\end{align}
We recall that
\begin{align*}
&E |\omega|\int_{\Omega} \rho^2 \Big(\frac{d \mathcal{U}_3}{d x_3} \Big)^2 \,dx_3+ \sum_{\alpha=1}^2 EI_\alpha\int_{(0, L)}  \rho^4 \Big(\frac{d^2 \mathcal{U}_\alpha}{dx_3^2}\Big)^2\,dx_3 +\frac{K\mu}{2}\int_{(0, L)}  \rho^4 \Big(\frac{d \mathcal{R}_3}{dx_3}\Big)^2 dx_3\\
= & |\omega|\int_{(0,L)} \rho^2 f\cdot {\cal U}\,dx_3- \sum_{\alpha=1}^2I_\alpha\int_{(0,L)}  \rho^4 g_\alpha {d{\cal U}_\alpha\over dx_3}\,dx_3+(I_1+I_2)\int_{(0, L)} \rho^4 g_3{\cal R}_3  dx_3.
\end{align*}
Finally we obtain
$$\lim_{\eps \to 0}\frac{\mathcal{E}(u^\eps)}{\eps^4}= \int_{\Omega} \Big\{\lambda Tr(T) Tr(T) + \sum_{i,j=1}^3 2\mu T_{ij} T_{ij}\Big\} dX_1dX_2dx_3,$$
which gives us the convergence of the rescaled energy to the total energy of the problems \eqref{bending}, \eqref{compression} and \eqref{equation R3} as $\eps$ goes to zero.
\end{proof}

Now we can deduce the strong convergences of the fields of the displacement decomposition using the strong convergence of the energy. In view of the weak convergence of the 
symmetric gradient \eqref{weak grad}, the strict convexity of the elastic energy implies that the convergence of the symmetric gradient is strong
\begin{equation}\label{strong sym grad}
 \frac{1}{\eps}\Pi_\eps\big( (\nabla u^\eps)_{\mathcal{S}}\big)  \rightarrow T \textrm{ strongly in } [L^2_\rho(\Omega)]^9.
\end{equation}
As a consequence we get 
\begin{align}\label{strong T33}
\frac{\rho}{\eps}\Pi_\eps \big( (\nabla u^\eps)_{\mathcal{S}}\big)_{33} &=  \frac{\rho}{\eps}\frac{d\mathcal{U}^\eps_3}{dx_3} -\rho \rho_\eps X_1\frac{d \mathcal{R}^\eps_2}{dx_3}+\rho \rho_\eps X_2\frac{d \mathcal{R}^\eps_1}{dx_3} + \frac{\rho}{\eps} \Pi_\eps \Big(\frac{\partial \bar{u}^\eps_3}{\partial x_3}\Big)\nonumber\\
&\rightarrow T_{33}=\rho \frac{d\mathcal{U}_3}{dx_3} - \rho^2X_1\frac{d \mathcal{R}_2}{dx_3}+\rho^2X_2\frac{d \mathcal{R}_1}{dx_3} \textrm{ strongly in } L^2(\Omega).
\end{align}
Moreover, using $\ds \int_{\omega} \Pi_\eps(\bar{u}^\eps_3) \, dX_1 dX_2 = \int_{\omega} X_\alpha\Pi_\eps(\bar{u}^\eps_3) \, dX_1 dX_2=0$, for $\alpha\in\{1,2\}$, and taking into account
convergence \eqref{Eq 427-0} we may deduce from
\eqref{strong T33} that
\begin{equation}\label{strong dU3dR3}
 \frac{\rho}{\eps}\frac{d\mathcal{U}^\eps_3}{dx_3}  \rightarrow \rho \frac{d\mathcal{U}_3}{dx_3}, \quad  \rho^2 \frac{d \mathcal{R}^\eps_\alpha}{dx_3} \rightarrow 
\rho^2\frac{d \mathcal{R}_\alpha}{dx_3}, \textrm{ strongly in } L^2(0,L), \, (\alpha=1,2),
\end{equation}
as $\eps$ tends to zero.
Then, in view of the weak convergences \eqref{weak2} and \eqref{weak3}, \eqref{strong dU3dR3} implies that
\begin{align}
&\frac{1}{\eps} \mathcal{U}^\eps_{3} \rightarrow \mathcal{U}_3 \textrm{ strongly in } H^1_\rho(0,L), \label{strong U3}\\
&\mathcal{R}^\eps_\alpha \rightarrow \mathcal{R}_\alpha\textrm{ strongly in } H^1_{\rho^2}(0,L), \quad \textrm{for } \alpha=1,2. \label{strong R12}
\end{align}
Moreover, from \eqref{weak30} and \eqref{strong R12} we have 
$$ \mathcal{U}^\eps_{\alpha} \rightarrow \mathcal{U}_\alpha \textrm{ strongly in } H^1_\rho(0,L), \textrm{ for } \alpha =1,2.$$
Hence, due to the decomposition \eqref{comp disp} and the previous strong convergences we deduce
\begin{align*}
\Pi_\eps (u^\eps_\alpha)  &\rightarrow \mathcal{U}_\alpha  \textrm{ strongly in } H^1_\rho(\Omega), \textrm{ for } \alpha =1,2.\\
\frac{1}{\eps}\Pi_\eps (u^\eps_3)&\rightarrow \mathcal{U}_{3}  - \rho  X_1{d\mathcal{U}_1\over dx_3} - \rho  X_2{d\mathcal{U}_2\over dx_3} \textrm{ strongly in } H^1_\rho(\Omega).
\end{align*}
We also have
$$\frac{1}{\eps^2}\gamma_{\alpha\beta}\big(\Pi_\eps  \bar{u}^\eps \big)\rightarrow \gamma_{\alpha\beta}( \bar{u})\;  \textrm{ strongly in } L^2(\Omega), \textrm{ for } \alpha,\beta =1,2.$$
We recall that the warping functions satisfy \eqref{warping cond1}. Then from the 2d Korn inequality we derive
$$
\sum_{\alpha=1}^2\Big\|\frac{1}{\eps^2}\Pi_\eps \big(\bar{u}^\eps_\alpha\big)-\bar{u}_\alpha\Big\|_{L^2(\Omega)}+\sum_{\alpha\beta=1}^2\Big\|\frac{1}{\eps^2}{\partial \Pi_\eps \big(\bar{u}^\eps_\alpha\big)\over \partial X_\beta}-{\partial\bar{u}_\alpha\over \partial X_\beta}\Big\|_{L^2(\Omega)}\le C\sum_{\alpha\beta=1}^2\Big\|\frac{1}{\eps^2}\gamma_{\alpha\beta}\big(\Pi_\eps  \bar{u}^\eps \big)-\gamma_{\alpha\beta}( \bar{u})\Big\|_{L^2(\Omega)}.
$$
That leads to
$$ \frac{1}{\eps^2}\Pi_\eps \big(\bar{u}^\eps_\alpha\big)\rightarrow \bar{u}_\alpha\;  \textrm{ strongly in } L^2((0,L); H^1(\omega)), \textrm{ for } \alpha =1,2.$$

\section{Conclusion}
In this last  section we summarize the results obtained in the previous sections.

\begin{theorem} Let $u^\eps$ be the solution of the elasticity problem \eqref{var form}. Under the assumptions \eqref{applied forces}-\eqref{weak conv forces} on the applied 
forces, the sequence $\{u^\eps\}$ satisfies the following convergences
\begin{align*}
\Pi_\eps (u^\eps_\alpha)  &\rightarrow \mathcal{U}_\alpha  \textrm{ strongly in } H^1_\rho(\Omega), \textrm{ for } \alpha =1,2,\\
\frac{1}{\eps}\Pi_\eps (u^\eps_3)&\rightarrow \mathcal{U}_{3}  - \rho  X_1{d\mathcal{U}_1\over dx_3} - \rho  X_2{d\mathcal{U}_2\over dx_3}  \textrm{ strongly in } H^1_\rho(\Omega),
\end{align*}
where $\mathcal{U}_\alpha$ is the solution of the bending problem \eqref{bending} and $\mathcal{U}_{3}$ is the weak solution of the stretching problem \eqref{compression}.
Moreover, we have
\begin{equation*}
 \frac{1}{\eps}\Pi_\eps\big( \gamma_{ij}(u^\eps)\big)  \rightarrow T_{ij} \textrm{ strongly in } L^2_\rho(\Omega), \textrm{ for } i,j =1,2,3 .
\end{equation*}
where
\begin{align*}
T_{11}&= T_{22} = -\nu T_{33},\qquad T_{33} = \rho \frac{d\mathcal{U}_3}{dx_3} - \rho^2X_1\frac{d^2 \mathcal{U}_1}{dx_3^2}-\rho^2X_2\frac{d^2\mathcal{U}_2}{dx_3^2},\\
T_{12}& = 0, \qquad T_{13}=\Big(-X_2+ \frac{\partial \chi}{\partial X_1}\Big) \frac{\rho^2}{2}\frac{d \mathcal{R}_3}{dx_3}, \qquad T_{23}=\Big(X_1+ \frac{\partial \chi}{\partial X_2}\Big) \frac{\rho^2}{2}\frac{d \mathcal{R}_3}{dx_3},
\end{align*}
with  $\chi \in H^1(\omega)$ is the solution of the torsion problem \eqref{torsion problem} and $\mathcal{R}_3$ the weak solution of \eqref{equation R3}.
\end{theorem}

 \bigskip

\noindent{\bf References.}

\end{document}